\documentclass{article}

\usepackage{amsfonts, amsmath, amssymb, amsthm, amscd}
\usepackage{ascmac}
\usepackage{verbatim}
\usepackage[dvips]{graphicx}
\usepackage{graphics}
\usepackage[all,2cell]{xypic}
\objectmargin+{1mm}
\labelmargin+{0.8mm}
\SelectTips{cm}{12}

\theoremstyle{plain}  
\newtheorem{thm}{Theorem}[section]
\newtheorem{con}[thm]{Conjecture}
\newtheorem{cor}[thm]{Corollary}
\newtheorem{lem}[thm]{Lemma}
\newtheorem{prop}[thm]{Proposition}

\theoremstyle{definition}

\newtheorem{df}[thm]{Definition}
\newtheorem{ex}[thm]{Example}

\newtheorem{nt}[thm]{Notations}

\newtheorem{rem}[thm]{Remark}
\newtheorem{rev}[thm]{Review}

\newtheorem{lemdf}[thm]{Lemma-Definition}
\newtheorem{cordf}[thm]{Corollary-Definition}
\newtheorem*{remark}{Remark}

\theoremstyle{remark}
\newtheorem*{claim}{Claim}

\DeclareMathOperator{\id}{id}
\DeclareMathOperator{\isoto}{\overset{\scriptstyle{\sim}}{\to}}
\DeclareMathOperator{\rinf}{\rightarrowtail}
\DeclareMathOperator{\linf}{\leftarrowtail}
\DeclareMathOperator{\rdef}{\twoheadrightarrow}

\DeclareMathOperator{\rinc}{\hookrightarrow}
\DeclareMathOperator{\linc}{\hookleftarrow}

\newcommand{\onto}[1]{\stackrel{#1}{\to}}
\newcommand{\ssm}{\smallsetminus}
\renewcommand{\coprod}{\sqcup}
\renewcommand{\leqq}{\leq}
\renewcommand{\geqq}{\geq}

\newcommand{\Ext}{\operatorname{Ext}}
\newcommand{\Ker}{\operatorname{Ker}}
\newcommand{\im}{\operatorname{Im}}
\newcommand{\coker}{\operatorname{Coker}}

\newcommand{\Cone}{\operatorname{Cone}}

\newcommand{\Typ}{\operatorname{Typ}}

\newcommand{\grade}{\operatorname{grade}}

\newcommand{\pd}{\operatorname{Projdim}}

\newcommand{\rank}{\operatorname{rank}}
\newcommand{\Spec}{\operatorname{Spec}}
\newcommand{\Supp}{\operatorname{Supp}}
\newcommand{\Tordim}{\operatorname{Tordim}}

\DeclareMathOperator{\Ob}{Ob}

\DeclareMathOperator{\op}{op}

\DeclareMathOperator{\HOM}{\mathcal{HOM}}

\DeclareMathOperator{\Tot}{Tot}

\DeclareMathOperator{\E}{E}

\DeclareMathOperator{\Homo}{H}

\newcommand{\bbN}{\operatorname{\mathbb{N}}}

\DeclareMathOperator{\ee}{\mathfrak{e}}
\DeclareMathOperator{\fE}{\mathfrak{E}}
\DeclareMathOperator{\ff}{\mathfrak{f}}
\DeclareMathOperator{\fF}{\mathfrak{F}}
\DeclareMathOperator{\fg}{\mathfrak{g}}
\DeclareMathOperator{\ii}{\mathfrak{i}}
\DeclareMathOperator{\pp}{\mathfrak{p}}
\DeclareMathOperator{\cSS}{\mathfrak{S}}

\DeclareMathOperator{\cA}{\mathcal{A}}
\DeclareMathOperator{\cB}{\mathcal{B}}
\DeclareMathOperator{\cC}{\mathcal{C}}
\DeclareMathOperator{\calD}{\mathcal{D}}
\DeclareMathOperator{\cE}{\mathcal{E}}
\DeclareMathOperator{\cF}{\mathcal{F}}
\DeclareMathOperator{\cG}{\mathcal{G}}
\DeclareMathOperator{\calH}{\mathcal{H}}

\DeclareMathOperator{\cM}{\mathcal{M}}

\DeclareMathOperator{\cO}{\mathcal{O}}
\DeclareMathOperator{\cP}{\mathcal{P}}

\DeclareMathOperator{\cS}{\mathcal{S}}

\DeclareMathOperator{\cW}{\mathcal{W}}
\DeclareMathOperator{\cX}{\mathcal{X}}
\DeclareMathOperator{\cY}{\mathcal{Y}}
\DeclareMathOperator{\cZ}{\mathcal{Z}}

\DeclareMathOperator{\Kos}{\mathbf{Kos}}

\DeclareMathOperator{\Cof}{\operatorname{Cof}}
\newcommand{\Cub}{\operatorname{\bf Cub}}

\newcommand{\red}{\operatorname{red}}
\newcommand{\tq}{\operatorname{tq}}
\newcommand{\nondeg}{\operatorname{nondeg}}

\UseAllTwocells
\def\sn{\smallskip\noindent}

\newcommand{\cf}{\textrm{cf.}\;}

\title{Higher $K$-theory of Koszul cubes}
\author{Satoshi Mochizuki}
\date{}

\begin{document}

\maketitle

\begin{abstract}
The main objective of this paper is 
to determine generators of 
the topological filtrations on 
the higher $K$-theory of a noetherian commutative ring with unit $A$. 
We introduce the concept of Koszul cubes and 
give a comparison theorem 
between the $K$-theory of Koszul cubes with 
that of topological filtrations. 
\end{abstract}

\section*{Introduction}

The following {\bf generator conjecture} 
is one of the significant problems 
in commutative algebra and algebraic $K$-theory. 
(For the relationship between 
the generator conjecture and Serre's intersection multiplicity conjecture 
\cite{Ser65}, 
please see the references \cite{Dut93}, \cite{Dut95}). 

\sn
Let $A$ be a commutative noetherian ring with unit and $p$ 
a natural number such that $0\leq p \leq \dim A$. 
Let $\cM^p_A$ denote the category of finitely generated $A$-modules $M$ 
whose support has codimension $\geqq p$ in $\Spec A$. 
Recall that 
a sequence of elements $f_1,\cdots,f_q$ in $A$ is 
said to be an {\bf $A$-regular sequence} if 
all $f_i$ are not unit elements and 
if $f_1$ is not a zero divisor of $A$ and 
if $f_{i+1}$ is not a zero divisor of $A/(f_1,\cdots,f_i)$ 
for any $1\leq i\leq q-1$.

\begin{con}[\bf Generator conjecture]
\label{con:Gencon}
For any commutative regular local ring $A$ and 
any natural number $0\leq p \leq \dim A$, 
the Grothendieck group 
$K_0(\cM^p_A)$ is generated by 
cyclic modules $A/(f_1,\cdots,f_p)$ where 
the sequence 
$f_1,\cdots,f_p$ forms an $A$-regular sequence. 
\end{con}

\sn
Conjecture~\ref{con:Gencon} is equivalent to 
Gersten's conjecture for $K_0$. 
Here is a statement of Gersten's conjecture for $K_n$:

\sn
{\it For any commutative regular local ring $A$ and 
natural numbers $n$, $p$, the canonical inclusion 
$\cM^{p+1}_A \rinc \cM_A^p$ induces the zero map on 
$K$-groups 
$$K_n(\cM_A^{p+1}) \to K_n(\cM_A^p)$$
where $K_n(\cM_A^i)$ denotes the $n$-th $K$-group of the abelian category $\cM_A^i$.} (See \cite{Ger73}).

\begin{remark}
\label{rem:His note}
Conjecture~\ref{con:Gencon} is known for the following cases.\\ 
$\mathrm{(i)}$ 
$A=B[[T_1,\cdots,T_n]]/(\Sigma T_i^2-\pi)$ 
where 
$B$ is a discrete valuation ring and 
$\pi$ is a prime element in $B$ 
or an unramified regular local ring $A$ 
by combining the result in \cite{Qui73}, \cite{Pan03}, \cite{GL87} 
and \cite{Lev85}. 
(Please see also related works \cite{Blo86} and \cite{CF68}).\\ 
$\mathrm{(ii)}$
$p=0$, $1$, $2$ and $\dim A$ 
(The results are classical for $p=0$ and $p=\dim A$ and for $p=1$ and $p=2$, please see the reference \cite{Smo87}).\\ 
\end{remark}

\sn
Let $\cM^p_A(p)$ denote the full subcategory of $\cM^p_A$ 
consisting of those $A$-modules $M$ of 
projective dimension $\leq p$. 
It is well-known that if $A$ is regular, then 
the canonical inclusion functor 
$\cM^p_A(p) \rinc \cM^p_A$ induces 
a homotopy equivalence on $K$-theory. (See \ref{thm:HM10}). 
For any endomorphism of $A$-module $\phi:F \to F$ between 
a finitely generated free $A$-module $F$, 
if we fix a basis $\alpha$ of $F$, 
then $\phi$ is represented by a square matrix $\Phi$. 
We write $\det_{\alpha}\phi$ or simply $\det\phi$ for $\det \Phi$ 
and call it the {\bf determinant of $\phi$} 
(with respect to $\alpha$). 
In connection with Conjecture~\ref{con:Gencon}, 
here is a corollary to the main theorem in this paper.

\begin{thm}
\label{thm:intromaincor} 
If $A$ is a local Cohen-Macaulay ring, 
then for any natural number $0\leq p\leq \dim A$, 
the Grothendieck group 
$K_0(\cM^p_A(p))$ is generated by modules of the form 
$$F/{<\im \psi_1,\cdots,\im\psi_p>}$$ 
where $F$ is a finitely generated free $A$-module 
and $\psi_k:F \to F$ is an $A$-module homomorphism 
such that 
the sequence 
$\det\psi_1,\cdots,\det\psi_p$ 
forms an $A$-regular sequence for any basis of $F$. 
\end{thm}

\begin{remark}
\label{rem:DHM85}
It is well-known that 
in general $K_0(\cM^3_A(3))$ is not generated by cyclic modules 
$A/(f_1,f_2,f_3)$ where $f_1,\ f_2,\ f_3$ forms a regular sequence. 
Please see the reference \cite{DHM85}, \cite{Lev88} and \cite{Bal09}. 
On the other hand, Smoke proved that 
for any $\dim A\geq p\geq 3$, 
$K_0(\cM_A^p(p))$ is generated by 
cyclic modules $A/(f_1,\cdots,f_r)$ $(r\geq p)$ where 
the sequence $f_1,\cdots,f_r$ forms an $A$-regular sequence. 
Please see the reference \cite[4.2]{Smo87}. 
\end{remark}

\sn
More generally, the main objective in the paper is 
to study topological filtrations on the higher $K$-theory of 
a commutative noetherian ring with unit. 
Inspired by the works of Gillet and Soul\'e \cite{GS87}, 
of Diekert \cite{Die86}, 
and of Grayson \cite{Gra92}, 
the main method in the paper is 
to replace certain full subcategories of modules 
with the category of cubes 
in the category of appropriate modules. 
More accurately, 
let us fix a commutative noetherian ring with unit $A$, 
a non-negative integer $q$ 
and a sequence $f_1,\ldots,f_p$ in $A$ 
such that for any bijection $\sigma$ on the set $S=\{1,\ldots,p\}$, 
$f_{\sigma(1)},\ldots,f_{\sigma(p)}$ is an $A$-regular sequence. 
We put $I=(f_1,\ldots,f_p)$ and $\ff_S=\{f_s\}_{s\in S}$. 
Let $\cM_A^{I}(q)$ denote the category 
of finitely generated $A$-modules $M$ such 
that $\pd_A M \leqq q$ 
and $\Supp M \subset V(I)$. 
(See Notations~\ref{nt:cM_A^f(p)}). 

\sn
A {\bf Koszul cube} $x$ associated with $f_1,\ldots,f_p$ 
is a contravariant 
functor from $[1]^{\times p}$ to 
the category of finitely generated projective $A$-modules $\cP_A$ 
where $[1]$ is the totally ordered set $\{0,1\}$ with the natural order 
$0 <1$ satisfying the condition that 
for each $1\leqq k \leqq p$ and 
$\ii=(i_1,\ldots,i_p)\in[1]^{\times p}$ such that $i_k=1$, 
$d_{\ii}^k:=x(\ii-\ee_k \to \ii)$ is injective and 
$\coker d_{\ii}^k$ is in $\cM_A^{f_k A}(1)$ 
where $\ee_k$ is the $k$-th unit vector. 
A morphism between Koszul cubes is just a natural transformation. 
We write $\Kos_A^{\ff_S}$ for 
the category of Koszul cubes associated with $f_1,\ldots,f_p$. 
(See Definition~\ref{df:Koszul cube df}). 
A Koszul cube $x$ associated with $f_1,\ldots,f_p$ 
is {\bf reduced} if for each $1\leqq k \leqq p$ and 
$\ii=(i_1,\ldots,i_p)\in[1]^{\times p}$ such that $i_k=1$, 
$f_k\coker d_{\ii}^k=0$. 
We write $\Kos_{A,\red}^{\ff_S}$ for 
the category of reduced Koszul cubes associated with 
the family $\ff_S=\{f_s\}_{s\in S}$. 
(See Lemma-Definition~\ref{lemdf:cMAfTfsp}). 
If we consider a Koszul cube $x$ as a multi-complex where 
$x_{(0,\ldots,0)}$ is in degree $(0,\ldots,0)$, 
we will take its total complex. 
(See Definition~\ref{df:Tot comp}). 
We will prove that for any Koszul cube $x$, 
$\Homo_k(\Tot x)=0$ for $k>0$ (See \ref{cor:charof adm cube}, 
\ref{cor:Kosisadm}). 
A morphism between Koszul cubes $f:x \to y$ is 
a {\bf quasi-isomorphism} if 
$\Homo_0\Tot f$ is an isomorphism. 
We denote the class of quasi-isomorphisms in 
$\Kos_A^{\ff_S}$ and 
$\Kos_{A,\red}^{\ff_S}$ by the same symbol $\tq$. 
The term ``Koszul" comes from the fact that 
the total complex of the cube of Example~~\ref{Koszul cube def} 
is just the usual Koszul complex associated with $\{f_s\}_{s\in S}$. 
We have the morphism of Waldhausen categories
$$\Homo_0\Tot:(\Kos_A^{\ff_S},\tq) \to (\cM_A^{\ff_S}(p),i)$$ 
where $i$ is the class of all isomorphisms. 
The next result is the comparison theorem referred to in the Abstract. 

\begin{thm}[\bf A part of Corollary~\ref{cor:main theorem}]
\label{thm:part local weight theorem}
The exact functor 
$\Homo_0\Tot:\Kos_A^{\ff_S} \to \cM_A^{\ff_S}(p)$ induces 
a homotopy equivalence on $K$-theory: 
$$K(\Kos_A^{\ff_S};\tq) \to 
K(\cM_A^{\ff_S}(p)).$$
\end{thm}

\sn 
When $A$ is a principle ideal domain, 
Theorem~\ref{thm:part local weight theorem} has been proven in \cite{PID}. 
To prove the theorem above, 
we develop a resolution theorem for 
Waldhausen categories. 
(See Theorem~\ref{resol thm}). 
The other ingredient of the proof is 
giving a quite elementary, 
but new algorithm of resolution process of modules 
by finite direct sums of typical Koszul cubes. 
(See Theorem~\ref{Koszul resol thm}). 
The second main theorem is the following:

\begin{thm}[\bf See Corollary~\ref{cor:Devissage 3}] 
\label{Introthm:dev3}
In the notation above, 
moreover if we assume that $A$ is regular, 
then the canonical inclusion functor 
$\iota:\Kos_{A,\red}^{\ff_S} \rinc \Kos_A^{\ff_S}$ 
induces the following homotopy equivalences on $K$-theory:
\begin{align*}
K(\Kos_{A,\red}^{\ff_S}) &\to K(\Kos_A^{\ff_S})\\
K(\Kos_{A,\red}^{\ff_S};\tq) &\to K(\Kos_A^{\ff_S};\tq).
\end{align*}
\end{thm}

\sn
To prove the theorem above, 
we will utilize the split fibration 
theorem~\ref{Split fibration theorem} 
which is a generalization of Lemma~3.3 in \cite{PID}. 
Theorem~\ref{Introthm:dev3} has the following application 
to Gersten's conjecture.

\begin{cor}
\label{cor:gercon}
Gersten's conjecture for a regular local ring $A$ 
is equivalent to the following assertion. 
For any $A$-regular sequence $\{f_s\}_{s\in S}$ in $A$, 
$\Homo_0\Tot:\Kos^{\ff_S}_{\red,A} \to \cM^{\# S-1}_A$ 
induces the zero maps on $K$-groups.
\end{cor}

\sn
The relationship between 
Gersten's conjecture and weight of the Adams operations on 
Koszul cubes, 
a higher analogue of generator conjecture 
will be studied in my subsequent papers 
by utilizing Corollary~\ref{cor:gercon}. 

\sn
By handling koszul cubes, 
in particular free Koszul cubes, 
we will be able to import linear algebra 
and combinatorial methods 
into our research for modules and (perfect) complexes 
in my forthcoming papers. 
On the other hand, 
since Waldhausen categories of Koszul cubes are not 
closed under taking the mapping cylinder functor, 
many standard theorems in Waldhausen $K$-theory, 
such as the generic fibration theorem and the approximation theorem 
in the literature do not apply directly. 
The paper is devoted to the concept of Koszul cubes, 
for example, homological algebra for cubes, 
and the fundamental technique of 
manipulating 
$K$-theory for Waldhausen category 
without assuming the factorization axiom. 

\sn
Now we give a guide for the structure of this paper. 
In section~\ref{sec:resol thm}, 
we give a resolution theorem 
for Waldhausen categories 
which is a generalization of 
Quillen's original one in \cite{Qui73}. 
In section~\ref{sec:semi-direct prod}, 
we develop the theory of semi-direct products 
of exact categories 
which is initiated in \cite{PID}. 
In section~\ref{sec:adm cubes}, 
we establish the theory about admissible cubes in 
an abelian category 
which is a categorical variant of 
the concept about regular sequences. 
We will calculate homology groups of 
the total complex associated with an admissible cube and 
utilizing this, 
we give several characterizations of admissibility. 
Finally we extend 
the notion of semi-direct products to 
that of multi semi-direct products 
of a family of exact categories. 
In section~\ref{sec:Koszul cubes}, 
we define Koszul cubes and 
by combining results in the previous sections, 
we describe the category of Koszul cubes by 
multi semi-direct products 
of the exact categories of pure weight modules. 
In section~\ref{sec:Koszul resolution theorem}, 
we give the algorithm of resolution process as mentioned above 
and as its corollary we get the first main result. 
In the final section, 
assuming the regularity of $A$, 
we will prove a d\'evissage theorem for Koszul cubes.

\sn
\textbf{Conventions.} 

\sn
$\mathrm{(1)}$ 
{\bf Set theory}\\ 
$\mathrm{(i)}$ 
Throughout this paper, 
we use the letter $S$ to denote a set.\\ 
$\mathrm{(ii)}$ 
For a positive integer $n$, 
we write $(n]$ for the set of integers $k$ such that $1\leqq k \leqq n$ 
and for a non-negative integer $m$, 
we denote the totally ordered set of integers 
$k$ such that $0\leqq k \leqq m$ by $[m]$.\\ 
$\mathrm{(iii)}$ 
For any set $S$, 
we write $\cP(S)$ for its {\bf power set}. 
Namely $\cP(S)$ is the set of all subsets of $S$. 
We consider $\cP(S)$ 
to be a partially ordered set under inclusion. 
A fortiori, $\cP(S)$ is a category.\\ 
$\mathrm{(iv)}$ 
For a finite set $S$, 
we denote the number of elements in $S$ by $\#S$.

\sn
$\mathrm{(2)}$ 
{\bf Commutative algebra}\\ 
$\mathrm{(i)}$ 
Throughout this paper, 
we use the letter $R$ (resp., $A$) 
to denote a commutative ring with $1$ 
(resp., commutative noetherian ring with $1$).\\ 
$\mathrm{(ii)}$ 
For any $R$, 
we write $R^{\times}$ for 
the group of units in $R$.\\
$\mathrm{(iii)}$ 
If ${\{f_s\}}_{s\in S}$ is a subset of $R$, 
we write $\ff_S$ for the ideal they generate. 
By convention, we set $\ff_{\emptyset}=(0)$.\\ 
$\mathrm{(iv)}$ 
A subring of $R$ is a subring with the same $1=1_R$.\\
$\mathrm{(v)}$ 
For any $R$, $A$, 
we let $\cP_R$ denote the category of finitely generated 
projective $R$-modules, 
and let $\cM_A$ denote the category of finitely generated $A$-modules. 

\sn 
$\mathrm{(3)}$ 
{\bf Category theory}\\
$\mathrm{(i)}$ 
Throughout the paper, 
we use the letters $\cC$ and $\cA$ 
to denote a category, 
an abelian category respectively.\\
$\mathrm{(ii)}$ 
For any category $\cC$, 
we denote the calss of objects in $\cC$ by $\Ob\cC$.\\
$\mathrm{(iii)}$ 
For any category $\cC$, 
$i\cC$ or shortly $i$ means the subcategory of all isomorphisms in $\cC$.\\
$\mathrm{(iv)}$ 
For categories $\cX$ and $\cY$, 
let us denote the (large) category of 
functors from $\cX$ to $\cY$ by $\HOM(\cX,\cY)$.\\ 
$\mathrm{(v)}$ 
For categories $\cX$ and $\cY$, 
a functor from $\cX$ to $\cY$, $j:\cX \to \cY$ is an {\bf inclusion functor} 
if it is fully faithful and 
the function between 
their classes of objects $j:\Ob\cX \to \Ob\cY$ is injective. 
We denote an inclusion functor by the arrow ``$\rinc$"

\sn
$\mathrm{(4)}$ 
{\bf Exact categories, Waldhausen categories and algebraic $K$-theory}\\ 
$\mathrm{(i)}$ 
Basically, 
for exact categories, 
we follow the notations in \cite{Qui73} and 
for algebraic $K$-theory of categories 
with cofibrations and weak equivalences, 
we follow the notations in \cite{Wal85}.\\ 
$\mathrm{(ii)}$ 
We denote an admissible monomorphism (resp. an admissible epimoprphism) by the arrow 
``$\rinf$" (resp. ``$\rdef$").\\ 
$\mathrm{(iii)}$ 
We call a category with 
cofibrations and weak equivalences 
a {\bf Waldhausen category}.\\ 
$\mathrm{(iv)}$ 
For a Waldhausen category $(\cX,w)$, 
we denote its $\cS$-construction by $w\cS_{\bullet}\cX$ 
and write $K(\cX;w)$ for the $K$-space $\Omega|w\cS_{\bullet}\cX|$. 
We also write $K(\cX)$ for $K(\cX;i)$.\\ 
$\mathrm{(v)}$ 
We say that 
a functor between exact categories 
(resp. categories with cofibrations) 
$f:\cX \to \cY$ {\bf reflects exactness} 
if for a sequence $x \to y \to z$ in $\cX$ 
such that $fx \to fy \to fz$ is an admissible exact sequence 
(resp. a cofibration sequence) in $\cY$, 
$x\to y \to z$ is an admissible exact sequence 
(resp. a cofibration sequence) in $\cX$.\\ 
$\mathrm{(vi)}$ 
For an exact category $\cE$, 
we say that its full subcategory $\cF$ is 
an {\bf exact subcategory} 
(resp. a {\bf strict exact subcategory}) 
if it is an exact category and 
the inclusion functor is exact (and reflects exactness).\\ 
$\mathrm{(vii)}$ 
Notice that as in \cite[p.321, p.327]{Wal85}, 
the concept of subcategories with cofibrations 
(resp. Waldhausen subcategories) is 
stronger than that of exact subcategories. 
Namely we say that $\cC'$ is a subcategory with cofibrations 
of a category with cofibration $\cC$ if 
a morphism in $\cC'$ is a cofibration in $\cC'$ if 
and only if it is a cofibration in $\cC$ and 
the quotient is in $\cC'$ (up to isomorphism). 
That is, the inclusion functor $\cC' \rinc \cC$ is 
exact and reflects exactness. 
For example, 
let $\cE$ be a non-semisimple exact category. 
Then $\cE$ with semi-simple exact structure 
is not 
a subcategory with cofibrations of $\cE$, 
but a exact subcategory of $\cE$.\\ 
$\mathrm{(viii)}$ 
Let $\cE$ be an exact category and $\cF$ a full subcategory of $\cE$. 
We say that $\cF$ is {\bf closed under kernels} 
({\bf of admissible epimorphisms}) 
if for any admissible exact sequence $x \rinf y \rdef z$ in $\cE$ 
if $y$ is 
isomorphic to object in $\cF$, 
then $x$ is also 
isomorphic to an object in $\cF$. (See \cite[II.7.0]{Wei12}).\\ 
$\mathrm{(ix)}$ 
We say that the class of morphisms $w$ 
in an exact category $\cE$ 
satisfies the {\bf cogluing axiom} 
if $(\cE^{\op},w^{\op})$ 
satisfies the gluing axiom.\\ 
$\mathrm{(x)}$ 
A pair of an exact category $\cE$ and a class of morphisms $w$ in $\cE$ is 
said to be a {\bf Waldhausen exact category} 
if $(\cE,w)$ and $(\cE^{\op},w^{\op})$ are Waldhausen categories.\\ 
$\mathrm{(xi)}$ 
For a Waldhausen category $(\cC,w)$, we write $w(\cC)$ 
if we wish to emphasis that $w$ is the class of weak equivalences in $\cC$. 
We write $\cC$ for $(\cC,w)$ when $w$ is 
the class of all isomorphisms in $\cC$.\\
$\mathrm{(xii)}$ 
An object $x$ in a Waldhausen category $(\cC,w)$ is {\bf $w$-trivical} 
if the canonical morphism $0 \to x$ is in $w$. 
We write $\cC^w$ for the full subcategory of $w$-trivial objects of $\cC$.\\
$\mathrm{(xiii)}$ 
For a Waldhausen category $\cC$ and subcategories with cofibrations $\cX$ and $\cY$ of $\cC$, 
let $E(\cX,\cC,\cY)$ denote the category with cofibrations of cofibration sequences 
$x\rinf y\rdef z$ 
in $\cC$ such that $x$ is in $\cX$ and $z$ is in $\cY$.

\sn
\textbf{Acknowledgements.} 
The author wishes to express his deep gratitude 
to Daniel R.~Grayson, 
Kei Hagihara, Akiyoshi Sannai 
and Seidai Yasuda 
for stimulating discussions, 
Kei-ichi Watanabe for 
instructing him in the 
direct proof of \ref{thm:TotstrictKos is 0sph} 
for the case of $\# S=2$ 
and Masanori Asakura, 
Toshiro Hiranouchi, 
Makoto Matsumoto, 
Nobuo Tsuzuki 
and Roozbeh Hazrat 
for inviting him to their universities. 
He also very thanks to the referees and Charles A.~Weibel for 
carefully reading a preprint version 
of this paper and giving innumerable and valuable comments 
to make the paper more readable and Amnon Neeman for 
encouraging him in finishing the work.

\section{A resolution theorem for Waldhausen categories}
\label{sec:resol thm}

In this section, 
we will prove a resolution theorem 
for Waldhausen categories 
by improving the proof 
for exact categories in \cite{Sta89}. 
This theorem is a generalization of 
Quillen's original one in \cite{Qui73}. 
Let $(\cX,w)$ be a Waldhausen category 
and $\cY$ a full subcategory 
of $\cX$ closed under extensions in $\cX$. 
In \ref{resol cond df} and 
\ref{strong resol cond df}, 
we will define the (strong) resolution conditions 
of the inclusion functor 
$\iota:\cY\rinc \cX$. 
In this situation, 
$\cY$ naturally becomes a 
Waldhausen subcategory 
and if $\cX$ is essentially small, 
then the canonical map induced by $\iota$, 
$K(\cY;w) \to K(\cX;w)$ 
is a homotopy equivalence \ref{resol thm}. 
In this section, from now on, 
let $(\cX,w)$ be a Waldhausen category. 

\begin{df}
\label{nt:multsystem}
Let $v$ be a class of morphisms in a category $\cC$. 
We say that 
$v$ is a {\bf multiplicative system} of $\cC$ if $v$ 
is closed under finite compositions 
and closed under isomorphisms. 
Namely\\ 
$\mathrm{(1)}$ 
if $\bullet \onto{f} \bullet \onto{g} \bullet$ are 
composable morphisms in $v$, then $gf$ is also in $v$, and\\ 
$\mathrm{(2)}$ 
all isomorphisms in $\cC$ are in $v$.\\ 
For a category $\cC$ and a multiplicative system $v$ of $\cC$, 
we define the simplicial subcategory $\cC(-,v)$ in $\HOM(-,\cC)$ 
$$[m]\mapsto \cC(m,v)$$
where $\cC(m,v)$ is the full subcategory of $\HOM([m],\cC)$ 
consisting of those functors 
which take values in $v$. 
For each $m$, 
we denote an object $x_{\bullet}$ in $\cC(m,v)$ by 
$$x_{\bullet}:x_0 \onto{i^x_0} x_1 \onto{i^x_1} x_2 \onto{i^x_2}\cdots 
\onto{i^x_{m-1}} x_m.$$
\end{df}

\begin{ex}
\label{ex:filobj}
(\cf \cite[1.1.4.]{Wal85}). 
For a category with cofibrations $(\cZ, \Cof\cZ)$ 
and each non-negative integer $m$, 
we can naturally make $F_m\cZ:=\cZ(m,\Cof\cZ)$ 
into a category with cofibrations. 
Here a morphism 
$a_{\bullet} \to a'_{\bullet}$ is defined to be a cofibration 
if for each $0\leqq j\leqq m$, 
$a_j \to a_j'$ and $a_j'\coprod_{a_j}a_{j+1} \to a_{j+1}'$ are 
cofibrations in $\cZ$. 
\end{ex}

\begin{ex}
\label{weak ladder df}
(\cf \cite[p.336 in the proof of 1.4.3.]{Wal85}). 
For $(\cX,w)$ and each non-negative integer $m$, 
we can make $\cX(m,w)$ 
into a category with cofibrations 
by defining the cofibrations 
to be term-wised cofibrations in $\cX$. 
\end{ex}

\begin{lem}
\label{lem:isomclosed}
Let $\calD$ be a full subcategory of a category $\cC$, 
$v$ a multiplicative system in $\cC$ 
and $m$ a non-negative integer. 
For each $x_{\bullet}$ 
in $\cC(m,v)$ if each $x_i$ is isomorphic to 
an object in $\calD$, 
then $x_{\bullet}$ is isomorphic to 
an object in $\calD(m,v)$. 
\end{lem}

\begin{proof}[\bf Proof] 
For each $x_j$, 
there are an object $y_j$ in $\cY$ and 
an isomorphism $\phi_j:x_j\isoto y_j$. 
We define the morphism $i^y_j:y_j\to y_{j+1}$ 
by the formula 
$i^y_j:=\phi_{j+1}i^x_j\phi_j^{-1}$. 
Since $v$ is a multiplicative system, 
$i^y_j$ is in $v$ and therefore 
$$y_{\bullet}:y_0\onto{i^y_0}y_1\onto{i^y_1}y_2
\onto{i^y_2}\cdots \onto{i^y_{m-1}}y_m$$ 
is an object in $\calD(m,v)$ and it is isomorphic to $x_{\bullet}$. 
\end{proof}

\begin{df}
\label{DDD notation}
Let $\cY$ be a full subcategory of $\cX$. 
$\cY$ is said to be 
{\bf closed under extensions} in $\cX$ 
if for a cofibration sequence 
$x \rinf y \rdef z$ in $\cX$, 
$x$ and $z$ are isomorphic to objects 
in $\cY$ respectively, 
then $y$ is also isomorphic to an object in $\cY$. 
In this case, 
$\cY$ is a Waldhausen category 
by declaring that a morphism $x \to y$ in $\cY$ 
is a cofibration in $\cY$ 
if it is a cofibration in $\cX$ 
and if $y/x$ is isomorphic to an object in $\cY$ 
and that a morphism $x \to y$ in $\cY$ 
is a weak equivalence in $\cY$ 
if it is a weak equivalence in $\cX$. 
From now on, 
let $\cY$ be a full subcategory of $\cX$ 
closed under extensions. 
\end{df}

\begin{rem}
\label{rem:closedunderextispreservecateq}
The extensional closed condition is 
preserved by equivalences as 
categories with cofibrations. 
That is, 
let us consider the 
commutative diagram of categories with cofibrations 
$$\footnotesize{\xymatrix{
\cZ \ar[r]^a \ar[d]^{\wr}_{i_{\cZ}} & \cW \ar[d]_{\wr}^{i_{\cW}}\\
\cZ' \ar[r]_{a'} & \cW'
}}$$
with both $i_{\cZ}$ and $i_{\cW}$ are fully faithful, 
essentially surjective and exact and reflect exactness. 
If $a:\cZ \rinc \cW$ is closed under extensions in $\cW$, 
then $a':\cZ'\rinc \cW'$ is also closed under extensions in $\cW'$.
\end{rem}

\begin{lem}
\label{lem:closedunderextensions}
In the case above, 
for any non-negative integer $m$, 
the inclusion functors 
$$\cY(m,w) \rinc \cX(m,w),$$ 
$$F_m\cY \rinc F_m\cX \text{ and}$$  
$$\cS_m\cY \rinc \cS_m\cX$$ 
are closed under extensions in $\cX(m,w)$, 
$F_m\cX$ and $\cS_m\cX$ respectively. 
\end{lem}

\begin{proof}[\bf Proof]
Let us consider a cofibration sequence
$$x_{\bullet} \rinf y_{\bullet} \rdef z_{\bullet}$$
in $\cX(m,w)$ or $F_m\cX$ 
and assume that $x_{\bullet}$ and $z_{\bullet}$ 
are isomorphic to 
objects in $\cY(m,v)$ or $F_m\cY$ respectively. 
Then by the definitions 
(see \ref{ex:filobj} or \ref{weak ladder df}), for 
each $0\leqq j\leqq m$, 
we have the cofibration sequence 
$$x_j\rinf y_j \rdef z_j$$
in $\cX$. 
Therefore by assumption, 
$y_j$ is isomorphic to an object in $\cY$. 
Now by \ref{lem:isomclosed}, 
we learn that $y_{\bullet}$ is isomorphic to an object 
in $\cY(m,v)$ or $F_m\cY$. 
This means that $\cY(m,w)$, 
$F_m\cY$ are closed under extensions 
in $\cX(m,v)$ or $F_m\cX$ respectively. 
Finally since we have the functorial equivalence 
$F_{m-1}\cX\isoto \cS_m\cX$ 
as categories with cofibrations, 
we notice that $\cS_m\cY$ is 
closed under extensions in $\cS_m\cX$ 
by \ref{rem:closedunderextispreservecateq}.
\end{proof}

\begin{df}
\label{relative category}
In the situation above, 
we can define the category $\cX_{\cY}$ as follows. 
The class of objects 
of $\cX_{\cY}$ is same as 
that of $\cX$. 
A morphism $x \to y$ in $\cX_{\cY}$ 
is a cofibration in $\cX$ such 
that $y/x$ is isomorphic to an object in $\cY$. 
One can easily prove that 
morphisms in $\cX_{\cY}$ are 
closed under compositions,
namely it is actually a category 
by virtue of 
the assumption \ref{DDD notation}.\\
Notice that
there is the natural inclusion functor 
$j:\Cof \cY \to \cX_{\cY}$. 
Here $\Cof \cY$ is 
the category of cofibrations in $\cY$. 
\end{df}

\begin{df}
\label{resol cond df} 
We say that 
the inclusion functor 
$\iota:\cY \rinc \cX$  
{\bf satisfies the resolution conditions} 
if it satisfies the following three conditions.\\ 
{\bf (Res 1)} 
$\cY$ is closed under extensions in $\cX$.\\ 
{\bf (Res 2)} 
For any object $x$ in $\cX$, 
there are an object $y$ in $\cY$ 
and a cofibration $x \rinf y$.\\ 
{\bf (Res 3)} 
For any cofibration sequence $x \rinf y \rdef z$ in $\cX$, 
if $y$ is in $\cY$, 
then $z$ is also in $\cY$.
\end{df}

\begin{lem}
\label{important contractibility}
{\rm (\cf\cite[Proof of 4.1.]{Gra87}, 
\cite[p.524]{Sta89})} 
If the inclusion functor $\iota:\cY \rinc \cX$ 
satisfies the resolution conditions, 
then $\cX_{\cY}$ is contractible. 
\end{lem}

\begin{proof}[\bf Proof] 
Since $\Cof\cY$ has the initial object, 
it is contractible. 
We intend to apply Quillen's Theorem A 
to $j:\Cof \cY \to \cX_{\cY}$ 
and then we will get the result. 
Fix an object $a$ in $\cX_{\cY}$ 
and objects $x$ and $y$ in $\cY$ such 
that there is a cofibration sequence $a \rinf x \rdef y$. 
Now we will prove that $a/j$ is contractible. 
To do so, 
consider an object $a \rinf b$ in $a/j$. 
Since $\cY$ is closed under extensions in $\cX$, 
in the following push out diagram below 
$$\xymatrix{ 
a \ar@{>->}[r] \ar@{>->}[d] \ar@{}[dr]|{\bigstar} 
& x \ar@{->>}[r] \ar@{>->}[d] 
& y \ar@{=}[d]\\ 
b \ar@{>->}[r] 
& b\coprod_a x \ar@{->>}[r] 
& y 
}$$ 
where the square $\bigstar$ is coCartesian, 
we can take $b\coprod_a x$ in $\cY$. 
Now there are the natural transformations 
$$((a \rinf b) 
\mapsto (a \rinf b)) \rinf ((a \rinf b) 
\mapsto (a \rinf b\coprod_a x)) \linf ((a \rinf b) 
\mapsto (a \rinf x))$$ 
between the identity functor 
and the constant functor 
$(a\rinf b) \mapsto (a\rinf x)$ 
on $a/j$. 
Therefore $a/j$ is contractible. 
\end{proof}

\begin{lem}
\label{inherit} 
If $\iota:\cY \rinc \cX$ satisfies the resolution conditions, 
then for each non-negative integer $n$, 
$\cS_n\cY \rinc \cS_n\cX$ also satisfies the resolution conditions.
\end{lem}

\begin{proof}[\bf Proof]
Since the filtered object categories 
$F_{n-1}\cX$ and $F_{n-1}\cY$ are 
equivalent to $\cS_n\cX$ and $\cS_n\cY$ respectively as 
categories with cofibrations, 
we just check that 
the inclusion functor 
$F_n\cY \rinc F_n\cX$ satisfies the resolution conditions. 
The condition $\textbf{(Res 1)}$ 
has been proven in \ref{lem:closedunderextensions}. 
We first check the condition \textbf{(Res 2)}. 
For a filtered object $x_0\rinf \ldots \rinf x_n$, 
we have an object $y_0$ in $\cY$ and a cofibration $x_0 \rinf y_0$ 
by the assumption \textbf{(Res 2)}. 
For each $k <n$, 
if we have a filtered object $y_0\rinf\ldots \rinf y_k$ in $F_k\cY$ and 
a cofibration $x_0\rinf\ldots\rinf x_k$ to $y$ in $F_k\cY$, 
then we have an object $y_{k+1}$ in $\cY$ and 
a cofibration $y_k \coprod_{x_k} x_{k+1} \rinf y_{k+1}$ 
by the assumption \textbf{(Res 2)} again. 
Therefore inductively, 
we can find a filtered object $y$ in $F_n\cY$ and 
a cofibration $x\rinf y$. 
Next we check the condition \textbf{(Res 3)}. 
For a cofibration sequence $x \rinf y \rdef z$ in 
$F_n\cX$, if $y$ is in $F_n\cY$, 
then by applying the assumption \textbf{(Res 3)} term-wisely, 
we notice that $z$ is also in $F_n\cY$. 
\end{proof}

\begin{df}
\label{strong resol cond df}
We say that 
the inclusion functor 
$\iota:\cY \rinc \cX$ 
{\bf satisfies the strong resolution conditions} 
if for any non-negative integer $m$, 
$\cY(m,w) \rinc \cX(m,w)$ 
satisfies the resolution conditions.
\end{df}

\begin{thm}[\bf Resolution theorem]
\label{resol thm}
In the notation above, 
if $\iota$ satisfies the strong resolution conditions 
and $\cX$ is essentially small, 
then the canonical map induced 
by $\iota$, 
$K(\cY;w) \to K(\cX;w)$ 
is a homotopy equivalence. 
\end{thm}

\begin{proof}[\bf Proof] 
We may assume that $\cX$ is small. 
By \cite[p.344, p.345 1.5.7]{Wal85}, 
we have the sequence of homotopy type of a fibration 
$$w\cS_{\bullet}\cY \overset{w\cS_{\bullet}\iota}{\to} 
w\cS_{\bullet}\cX \to w\cS_{\bullet}F_{\bullet}(\cX,\cY).$$ 
Fix non-negative integers $n$ and $m$. 
We have the following equalities. 
\begin{align*}
N_mw\cS_nF_{\bullet}(\cX,\cY) & \isoto 
f_{\bullet}(\cS_n\cX(m,w\cS_n\cX),\cS_n\cY(m,w\cS_n\cY))\\ 
& \isoto  f_{\bullet}(\cS_n(\cX(m,w)),\cS_n(\cY(m,w))) 
\end{align*}
where $f_{\bullet}$ denote the simplicial set 
of objects of $F_{\bullet}$ 
and for the definition $\cX(m,w)$ 
and so on see \ref{weak ladder df}. 
By the realization lemma 
\cite[Appendix A]{Seg74} or \cite[5.1]{Wal78}, 
and by replacing $\cX(m,w)$ and $\cY(m,w)$ 
with $\cX$ and $\cY$ respectively, 
we shall just check the following claim. 

\begin{claim} 
For a small category 
with cofibrations $\cX$ 
and $\iota:\cY \rinc \cX$ 
a full sub category closed under extensions. 
Assume that $\iota$ is satisfying 
the resolution conditions. 
Then for each non-negative integer $n$, 
$f_{\bullet}(\cS_n\cX,\cS_n\cY)$ is contractible.
\end{claim}

\sn
If $n=0$, 
this claim is trivial. 
For $n \geqq 1$, 
by \ref{inherit} and by replacing $\cS_n\cX$ and $\cS_n\cY$ with 
$\cX$ and $\cY$ respectively, 
we shall assume $n=1$. 
Now $f_{\bullet}(\cX,\cY)$ 
is just the nerve of $\cX_{\cY}$ 
in \ref{relative category} 
and therefore we get the result 
by \ref{important contractibility}. 
\end{proof}

\section{Semi-direct products of exact categories}
\label{sec:semi-direct prod}

In this section, 
we will establish the theory 
of semi-direct products of exact categories 
(with weak equivalences) 
which is a generalization of \cite[\S 3]{PID}. 
Let us start by preparing the general terminologies about cubes. 
Let $S$ be a set, $\cC$ a category, $\cA$ an abelian category 
and $R$ a commutative ring with $1$.

\begin{df}
\label{cube df} 
We define the {\bf category of $S$-cubes} in $\cC$ 
by 
$$\Cub^S(\cC):=\HOM(\cP(S)^{\op},\cC).$$
An object in $\Cub^S(\cC)$ is said to be an {\bf $S$-cube}. 
Let $x$ be an $S$-cube in $\cC$. 
For $T\in\cP(S)$ and $k\in T$, 
we denote $x(T)$ by $x_T$ and call it a {\bf vertex} of $x$ (at $T$) 
and we also write $d^{x,k}_T$ or 
shortly $d^k_T$ for $x(T\ssm\{k\} \rinc T)$ 
and call it 
a ({\bf $k$-direction}) {\bf boundary morphism} of $x$. 
\end{df}

\begin{rem}
\label{rem:cubentrem}
For a positive integer $n$, 
we have the canonical category isomorphism
$$\cP((n])\isoto [1]^{\times n},\ \ S \mapsto (\chi_S(k))$$ 
where $\chi_S$ is the 
{\bf characteristic function} associated with $S$. 
Namely $\chi_S(k)=1$ if $k$ is in $S$ and otherwise $\chi_S(k)=0$. 
Through the isomorphism above, 
we consider $(n]$-cubes to be 
contravariant functors from $[1]^{\times n}$ 
and call them {\bf $n$-cubes}. 
$\Cub^{(n]}$ is abbreviated to $\Cub^n$. 
\end{rem}

\begin{rem} 
For any abelian (resp. exact) category $\cC$ 
and a set $S$, 
$\Cub^S(\cC)$ is an abelian category 
(resp. exact category 
by defining the admissible exact sequences 
to be termwise admissible exact sequences in $\cC$).
\end{rem} 

\begin{rem} 
\label{rem:disjointindex}
For a pair of disjoint sets $S$ and $T$, 
we have the category isomorphism
$$\cP(S)\times\cP(T)\isoto \cP(S\coprod T),\ (U,V)\mapsto U\cup V$$
and by the exponential law, 
the isomorphism above induce the category isomorphism 
$$\Cub^{S\coprod T}(\cC)\isoto\Cub^S(\Cub^T(\cC)) .$$
Moreover if $\cC$ is an abelian (resp. exact) category, 
then the isomorphism above is an exact functor.
\end{rem} 

\begin{df}[\bf Homology of cubes] 
\label{cube homology}
Let us fix an $S$-cube $x$ in $\cA$. 
For each $k$, 
the {\bf $k$-direction $0$-th} 
(resp. {\bf $1$-th}) {\bf homology} of $x$ 
is an $S\ssm\{k\}$-cube in $\cA$ denoted by $\Homo_0^k(x)$ 
(resp. $\Homo_1^k(x)$) 
and defined by $\Homo_0^k(x)_T:=\coker d_{T\cup\{k\}}^k$ 
(resp. $\Homo_1^k(x)_T:=\Ker d_{T\cup\{k\}}^k$). 
\end{df} 

\sn
The following lemma is sometimes useful to deal with morphisms of cubes.

\begin{lem} 
\label{lem:genofcubemap} 
We have the following assertions.\\ 
$\mathrm{(1)}$ 
For any $S$-cube $x$, 
every $T$, $U\in\cP(S)$ such that 
$T \subset U$ and $U \ssm T$ is a finite set, 
the morphism $x(T \subset U)$ 
is described as compositions of boundary morphisms.\\ 
$\mathrm{(2)}$ 
Assume that $S$ is a finite set. 
For any $S$-cubes $x$, $y$ and a family of morphisms 
$f={\{f_T:x_T \to y_T\}}_{T\in\cP(S)}$ in $\cC$, 
$f:x \to y$ is a morphism of $S$-cubes in $\cC$ 
if and only if for any $T\in\cP(S)$ and $k\in T$, 
we have the equality $d_T^{y,k}f_T=f_{T\ssm\{k\}}d_T^{x,k}$. 
\end{lem}

\begin{ex}[\bf Typical cubes]
\label{Koszul cube def}
Assume that $S$ is a finite set and 
let $\ff_S=\{f_s\}_{s\in S}$ be a family of elements in $R$. 
The {\bf typical cube} associated 
with $\{f_s\}_{s\in S}$ 
is an $S$-cube in $\cP_R$ 
denoted by
$\Typ_R(\ff_S)$ 
and defined by 
$\Typ_R(\ff_S)_T=R$ and $d_T^{\Typ_R(\ff_S),t}=f_t$ for any 
$T\in \cP(S)$ and $t\in T$. 
\end{ex}

\sn
The following lemma is often used when we are dealing with cubes.
Its proof is very easy. 

\begin{lem}[\bf Cube lemma]
\label{cube lemma} 
For the diagram below in a category $\cC$:
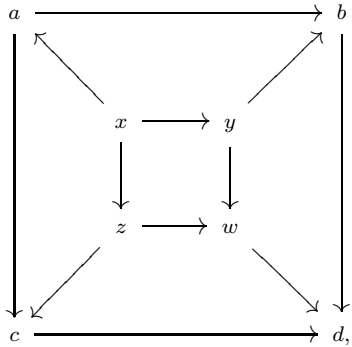
$$\footnotesize{\xymatrix{
a \ar[rrr] \ar[ddd] 
& & & b \ar[ddd]\\ 
& x \ar[r] \ar[d] \ar[ul] 
& y \ar[d] \ar[ur] 
& \\
& z \ar[dl] \ar[r] 
& w \ar[dr] \\
c \ar[rrr] & & & d,  
}}$$
assume that the morphism $\vec{wd}$ is a monomorphism  
{\rm (}resp. $\vec{xa}$ is an epimorphism{\rm )} 
and every squares except $xywz$ 
{\rm (}resp. $abdc${\rm )} 
are commutative. 
Then $xywz$ 
{\rm (}resp. $abdc${\rm )} 
is also commutative.
\end{lem}

\begin{df}
\label{semi-direct product nt}
Let $\cE$ and $\cF$ be full strict exact subcategories 
of $\cA$. 
We define the category 
$\cF \ltimes \cE$ 
as follows. 
$\cF \ltimes \cE$ 
is the full subcategory of $\Cub^1(\cE)$ 
of those morphisms 
$x_1 \to x_0$ in $\cE$ 
which are monomorphisms with $\cA$-cokernel in $\cF$. 
\end{df}

\begin{prop}
\label{semidiret is exact}
In the notations above, 
if $\cF$ satisfies either condition 
$\mathrm{(1)}$ or $\mathrm{(2)}$ below, 
then $\cF \ltimes \cE$ is 
a full strict exact subcategory of $\Cub^1(\cA)$. 
Moreover $\Homo_0:\cF \ltimes \cE \to \cF$ is exact.\\ 
$\mathrm{(1)}$ 
$\cF$ is closed under extensions, 
that is, 
for an exact sequence $ a\rinf b \rdef c$ in $\cA$, 
if $a$ and $c$ are isomorphic to objects in $\cF$ respectively, 
then $b$ is also isomorphic to an object in $\cF$.\\ 
$\mathrm{(2)}$ 
$\cF$ is closed under admissible sub- and quotient objects, 
that is, 
for an exact sequence $ a\rinf b \rdef c$ in $\cA$, 
if $b$ is isomorphic to an object in $\cF$, 
then $a$ and $c$ are also isomorphic to objects in $\cF$ respectively.
\end{prop}

\begin{proof}[\bf Proof]
We may assume that $\cE$ and $\cF$ are 
closed under isomorphisms in $\cA$. 
That is, 
if $a$ is an object in $\cA$ 
which is isomorphic to an object in $\cE$ (resp. $\cF$), 
then $a$ is also in $\cE$ (resp. $\cF$). 
We declare that a sequence 
$x \to y \to z$ in $\cF \ltimes \cE$ 
is an admissible exact sequence if it is exact in $\Cub^1(\cA)$. 
Obviously a split short exact sequence 
in $\cF \ltimes \cE$ is an admissible exact sequence. 
We need to prove 
that the class of admissible monomorphisms 
(resp. admissible epimorphisms) 
is closed under compositions 
and co-base change 
(resp. base change) 
along arbitrary morphisms. 
We just check for the admissible monomorphisms case. 
To prove for the admissible epimorphisms case is similar. 
For a pair of composable admissible monomorphisms 
$x \rinf y \rinf z$, 
by the snake lemma in $\Cub^1(\cA)$, 
we have the short exact sequence
$$y/x \rinf z/x \rdef z/y$$
in $\Cub^1(\cA)$. 
Applying the snake lemma to the following diagram
$$\footnotesize{\xymatrix{ 
(y/x)_1 \ar@{>->}[r] \ar[d]_{d^{y/x}} 
& (z/x)_1 \ar[d]^{d^{z/x}} \ar@{->>}[r] 
& (z/y)_1 \ar[d]^{d^{z/y}} 
& \\ 
(y/x)_0 \ar@{>->}[r] 
& (z/x)_0 \ar@{->>}[r] 
& (z/y)_0 , 
}}$$ 
we learn that $d^{z/x}$ is a monomorphism in $\cA$. 
Now let us assume 
the condition $\mathrm{(1)}$ (resp. $\mathrm{(2)}$). 
Then by 
considering the following short exact sequence in $\cA$ 
$$\Homo_0(y/x) \rinf \Homo_0(z/x) \rdef \Homo_0(z/y),$$ 
$$\text{(resp. $\Homo_0(x) \rinf \Homo_0(z) \rdef \Homo_0(z/x),$\ )}$$ 
we notice that $\Homo_0(z/x)$ is actually in $\cF$. 
Therefore the short exact sequence 
$$x \rinf z \rdef z/x$$ 
is an admissible exact sequence in $\cF \ltimes \cE$. 
Hence the class of admissible monomorphisms is closed under compositions. 
Next let us consider morphisms 
$y \leftarrow x \rinf z$ in $\cF \ltimes \cE$. 
Consider the coproduct $y\coprod_x z$ in $\Cub^1(\cA)$. 
We have the following pushout diagram
\begin{equation}
\label{equ:pushout}
\footnotesize{\xymatrix{
x \ar@{>->}[r] \ar[d] & z \ar@{->>}[r] \ar[d] & z/x \ar[d]\\
y \ar@{>->}[r] & y\coprod_x z \ar@{->>}[r] & z/x
}}
\end{equation}
in $\Cub^1(\cA)$. 
Then since the class of admissible monomorphisms in $\cE$ 
is closed under the co-base change 
by arbitrary morphisms, 
${(y\coprod_x z)}_i$ is in $\cE$ for $i=0$, $1$. 
Applying the snake lemma to the following diagram 
$$\footnotesize{\xymatrix{
y_1 \ar@{>->}[r] \ar[d]_{d^{y}} 
& (y\coprod_x z)_1 \ar[d]^{d^{y\coprod_x z}} \ar@{->>}[r] 
& (z/x)_1 \ar[d]^{d^{z/x}} 
& \\
y_0 \ar@{>->}[r] 
& (y\coprod_x z)_0 \ar@{->>}[r] 
& (z/x)_0, 
}}$$
we learn that $d^{y\coprod_x z}$ is a monomorphism in $\cA$ 
and 
by applying the functor $\Homo_0$ to 
the pushout diagram (\ref{equ:pushout}) above, 
we get the following commutative diagram 
$$\footnotesize{\xymatrix{
\Homo_0(x) \ar@{>->}[r] \ar[d] \ar@{}[dr]|{\bigstar}
& \Homo_0(z) \ar[d] \ar@{->>}[r] 
& \Homo_0(z/x) \ar@{=}[d] 
& \\ 
\Homo_0(y) \ar@{>->}[r] 
& \Homo_0(y \coprod_x z) \ar@{->>}[r] 
& \Homo_0(z/x) 
}}$$ 
where the horizontal lines are short exact sequences in $\cA$. 
Thus $\bigstar$ is coCartesian 
and since the class of admissible monomorphisms in 
$\cF$ is closed under co-base change by arbitrary morphisms, 
$\Homo_0(y \coprod_x z)$ is in $\cF$. 
We conclude that $y\coprod_x z$ is in $\cF \ltimes \cE$ and 
$$y \rinf y\coprod_x z \rdef z/x$$ 
is an admissible exact sequences in $\cF \ltimes \cE$. 
\end{proof}

\begin{df}
\label{semi-direct we nt}
In the situation \ref{semidiret is exact}, 
moreover assume that $\cF$ has the subcategory $w=w(\cF)$ 
containing the class of isomorphisms in $\cF$. 
We define the subcategory $tw(\cF\ltimes\cE)$ 
of $\cF\ltimes\cE$ as follows. 
A morphism $f:x \to y$ in $\cF\ltimes\cE$ 
is in $tw(\cF\ltimes\cE)$ 
if and only if 
$\Homo_0(f):\Homo_0(x) \to \Homo_0(y)$ 
is in $w(\cF)$. 
Then every isomorphism in $\cF\ltimes\cE$ 
is in $tw(\cF\ltimes\cE)$. 
\end{df}

\begin{prop}
\label{semi-direct we}
In the notation above, 
if $w(\cF)$ satisfies the gluing 
(resp. cogluing, saturational, extensional) axiom, 
then $tw(\cF\ltimes\cE)$ also does. 
In particular if $(\cF,w)$ is a Waldhausen exact category, 
then $(\cF\ltimes\cE,tw(\cF\ltimes\cE))$ is also 
a Waldhausen exact category and the functor 
$\Homo_0:(\cF\ltimes\cE,tw(\cF\ltimes\cE)) 
\to (\cF,w)$ 
is a morphism of Waldhausen categories.
\end{prop}

\begin{prop}
\label{semi-direct homotopy equivalence}
In the notation \ref{semi-direct we}, 
moreover assume that $(\cF,w)$ 
is a Waldhausen category 
and $\cF$ is contained in $\cE$, 
then the functor $\Homo_0$ 
induces a homotopy equivalence on $K$-theory:
$$ 
K(\Homo_0):K(\cF\ltimes\cE;tw) 
\to K(\cF;w).
$$  
\end{prop}

\begin{proof}[\bf Proof]
Define the morphism of Waldhausen categories 
$s:(\cF,w) \to (\cF\ltimes\cE,tw)$ 
by $x \mapsto [0 \to x]$. 
Then obviously 
we have the equality 
$\Homo_0\circ s=\id$. 
Moreover there is the natural weak equivalence 
$\id \to s\circ\Homo_0$ 
defined by the canonical morphisms 
$$\begin{bmatrix}
\xymatrix{ 
x_1 \ar[d]_{\scriptstyle{d^x}}\\ 
x_0 
} 
\end{bmatrix} 
\begin{matrix} 
\xymatrix{ 
\to \ar@{}[d]\\ 
\to\\ 
} 
\end{matrix} 
\begin{bmatrix}
\xymatrix{ 
0 \ar[d] \\ 
\Homo_0(x) 
}
\end{bmatrix}$$
for any object $x$ in $\cF\ltimes\cE$. 
Therefore by \cite[p.330 1.3.1]{Wal85}, 
we learn that $wS_{\bullet}\Homo_0$ 
is a homotopy equivalence. 
\end{proof}

\sn
Next let $\cE_i$ and $\cF_i$ $(i=1$, $2)$ 
be full strict exact subcategories 
of $\cA$. 
Moreover we suppose that $\cF_i$ 
satisfies either condition $\mathrm{(1)}$ or $\mathrm{(2)}$ 
in \ref{semidiret is exact}. 
Then by \ref{semidiret is exact}, 
$\cF_i\ltimes\cE_i$ is an exact category.

\begin{prop}
\label{closed conditions 1}
Assume that the inclusion functors 
$\cE_1 \rinc \cE_2$ and $\cF_1 \rinc \cF_2$ 
are closed under extensions. 
That is, 
for an admissible exact sequence 
$$x \rinf y \rdef z$$ 
in $\cE_2$ {\rm (}resp. $\cF_2${\rm )} 
if $x$ and $z$ are isomorphic to objects 
in $\cE_1$ {\rm(}resp. $\cF_1${\rm )}, 
then $y$ is also 
in $\cE_1$ {\rm(}resp. $\cF_1${\rm )}. 
Then $\cF_1\ltimes\cE_1 \rinc \cF_2\ltimes\cE_2$ 
is also closed under extensions. 
\end{prop}

\begin{prop}
\label{closed conditions 2}
Assume that the inclusion functors 
$\cE_1 \rinc \cE_2$ 
and $\cF_1 \rinc \cF_2$ 
are closed under taking kernels of admissible epimorphisms. 
That is, 
for a short exact sequence 
$$x \rinf y \rdef z$$ 
in $\cA$ if $z$ is isomorphic to an object in $\cE_2$ 
{\rm (}resp. $\cF_2${\rm )} 
and $y$ is isomorphic to an object 
in $\cE_1$ 
{\rm (}resp. $\cF_1${\rm )}, 
then $x$ is also isomorphic 
to an object in $\cE_1$ 
{\rm (}resp. $\cF_1${\rm )}. 
Then $\cF_1\ltimes\cE_1 \rinc \cF_2\ltimes\cE_2$ 
is also closed under taking kernels of admissible epimorphisms. 
\end{prop}

\begin{proof}
[\bf Proof of \ref{closed conditions 1} 
and \ref{closed conditions 2}.] 
Let us consider the short exact sequence below 
in $\Cub^1(\cA)$ 
$$\footnotesize{\xymatrix{ 
x_1 \ar@{>->}[r] \ar[d]_{d^x} 
& y_1 \ar@{->>}[r] \ar[d]^{d^y} 
& z_1 \ar[d]^{d^z} 
&\\ 
x_0 \ar@{>->}[r] 
& y_1 \ar@{->>}[r] 
& z_1 . 
}}$$ 
If $d^z$ and $d^x$ are monomorphisms 
(resp. $d^y$ is a monomorphism), 
then $d^y$ (resp. $d^x$) is also. 
In this case, 
observing the $3\times 3$ commutative diagram below 
$$\footnotesize{\xymatrix{ 
x_1 \ar@{>->}[r] \ar[d]_{d^x} 
& y_1 \ar@{->>}[r] \ar[d]^{d^y} 
& z_1 \ar[d]^{d^z} 
&\\ 
x_0 \ar@{>->}[r] \ar@{->>}[d] 
& y_1 \ar@{->>}[r] \ar@{->>}[d] 
& z_1 \ar@{->>}[d] 
& \\ 
\Homo_0(x) \ar@{>->}[r] 
& \Homo_0(y) \ar@{->>}[r] 
& \Homo_0(z)  , 
}}$$ 
we learn that if 
the condition \ref{closed conditions 1} 
(resp. \ref{closed conditions 2}) is verified 
and if $z$ is isomorphic to an object 
in $\cF_1\ltimes\cE_1$ 
(resp. $\cF_2\ltimes\cE_2$) 
and if $x$ (resp. $y$) is isomorphic to an object 
in $\cF_1\ltimes\cE_1$, 
then $y$ (resp. $x$) 
is also isomorphic 
to an object in $\cF_1\ltimes\cE_1$.  
\end{proof}

\begin{rem} 
\label{rem:closed condition} 
The assertions 
\ref{closed conditions 1},  \ref{closed conditions 2} 
and dual of \ref{closed conditions 2} imply 
the following statements.:\\ 
$\mathrm{(1)}$ 
Assume that the inclusion functors 
$\cE_1 \rinc \cE_2$ 
and $\cF_1 \rinc \cF_2$ 
are closed under admissible sub- and quotient objects, 
then $\cF_1\ltimes \cE_1$ is 
also closed in $\cF_2\ltimes \cE_2$.\\ 
$\mathrm{(2)}$ 
Let 
$\cE$, $\cF$ be full subcategories of $\cA$ 
closed under extensions in $\cA$. 
Then $\cF\ltimes\cE$ is closed under extensions 
in $\Cub^1(\cA)$.
\end{rem}

\begin{prop}
\label{prop:projobjinsemidirectprod}
In the notation above, 
moreover let us assume the following two conditions.\\ 
$\mathrm{(1)}$ 
$\cF_2$ is a full subcategory of $\cE_2$.\\ 
$\mathrm{(2)}$ 
Every object in $\cE_1$ is a projective object in $\cE_2$ and 
every object in $\cF_1$ is a projective object in $\cF_2$.\\ 
Then all objects in $\cF_1 \ltimes \cE_1$ 
are projective objects in $\cF_2 \ltimes \cE_2$. 
\end{prop}

\begin{proof}[\bf Proof] 
Let us consider the left diagram in $\cF_2\ltimes \cE_2$ below 
$$\footnotesize{\xymatrix{
x \ar[rd]^f \ar@{-->}[d]_s &\\
y \ar@{->>}[r]_t & z  ,
}\ \ \ \ \ 
\xymatrix{
\Homo_0(x) \ar[rd]^{\Homo^0(f)} \ar@{-->}[d]_{\sigma} &\\
\Homo_0(y) \ar@{->>}[r]_{\Homo_0(t)} & \Homo_0(z)}}$$
where $x$ is an object in $\cF_1\ltimes \cE_1$ and $t$ is 
an admissible epimorphism. 
Then by applying $\Homo_0$ to the diagram, 
we get the right diagram in $\cF_2$ above. 
Since $\Homo_0(x)$ is a projective object in $\cF_2$, 
there is a morphism $\sigma:\Homo_0(x) \to \Homo_0(y)$ which makes the 
right diagram above commutative. 
\begin{claim}
There is a morphism $s':x \to y$ such that 
$\Homo_0(s')=\sigma$ and $ts'$ is chain homotopic to $f$. 
\end{claim}
\begin{proof}[\bf Proof of claim] 
Let us consider the left diagram of admissible exact sequences below
$$\footnotesize{\xymatrix{ 
x_1 \ar@{>->}[r]^{d^x} \ar@{-->}[d]_{s'_1} & 
x_0 \ar@{->>}[r]^{\pi^x} \ar@{-->}[d]_{s'_0} &
\Homo_0(x) \ar[d]^{\sigma}\\
y_1 \ar@{>->}[r]_{d^y} & 
y_0 \ar@{->>}[r]_{\pi^y} &
\Homo_0(y) ,
}\ \ \ \ \ 
\xymatrix{
x_1 \ar@{>->}[r]^{d^x} \ar[d]_{{(f-ts')}_1} & 
x_0 \ar@{->>}[r]^{\pi^x} \ar[d]^{{(f-ts')}_0} \ar@{-->}[ld]_h&
\Homo_0(x) \ar[d]^{0}\\
z_1 \ar@{>->}[r]_{d^z} & 
z_0 \ar@{->>}[r]_{\pi^z} &
\Homo_0(z) .
}}$$
Since $x_0$ is a projective object in $\cE_2$, 
we have a morphism $s'_0:x_0 \to y_0$ 
which makes the diagram above commutative. 
Therefore by the universality for the kernel of $d^y$, 
we also have a morphism $s'_1:x_1 \to y_1$ 
in the left commutative diagram above. 
Then we have the equalities $\Homo_0(f)=\Homo_0(t)\sigma=\Homo_0(ts')$. 
Therefore we have $\pi^z{(f-ts')}_0=0$. 
By the universality for the kernel of $\pi^z$, 
we have a morphism $h:x_0 \to z_1$ such that ${(f-ts')}_0=d^zh$. 
Since $d^z$ is a monomorphism, 
we also have the equality ${(f-ts')}_1=hd^x$. 
Hence we get the desired result.
\end{proof}
\sn
Since $x_0$ is a projective object in $\cE_2$, 
we have a morphism $u:x_0 \to y_1$ such that $t_1u=h$. 
$$\footnotesize{\xymatrix{ 
x_0 \ar@{-->}[d]_u \ar[rd]^h & \\
y_1 \ar@{->>}[r]_{t_1} & z_1  .
}}$$
We put 
$s_1:=s'_1+ud^x$ and 
$s_0:=s'_0+d^yu$. 
Then we can easily check that $s$ is a morphism of complexes 
and $f=ts$. 
\end{proof}

\begin{nt}
\label{nt:split fib thm}
Let 
$\calH \rinc \cG$ be strict exact subcategories of $\cA$. 
Assume that $\cG$ satisfies either condition 
$\mathrm{(1)}$ or $\mathrm{(2)}$ in \ref{semidiret is exact}. 
Moreover assume that $\cG$ has 
a class of weak equivalences $w\cG$ which 
satisfies the axioms of weak equivalences in \cite{Wal85}. 
We put $v:=w\cG\cap \calH$ and 
it is a class of weak equivalences in $\calH$. 
We define the new class of weak equivalences $lv(\cG\ltimes\calH)$ 
in $\cG\ltimes\calH$ as follows. 
A morphism $f:x \to y$ in $\cG\ltimes\calH$ 
is in $lv(\cG\ltimes\calH)$ 
if and only if $f_i:x_i \to y_i$ for $i=0$, $1$ is in $v\calH$. 
We call a morphism in $lv(\cG\ltimes\calH)$ 
a {\bf level weak equivalence}. 
\end{nt}

\sn
We can easily check that $(\cG\ltimes\calH,lv)$ 
is a Waldhausen category. 

\begin{prop}[\bf Abstract split fibration theorem]
\label{Split fibration theorem}
Let $\cG$, $\calH$, $w$ and $v$ be as in \ref{nt:split fib thm}. 
Then we have a split fibration sequences
$$K(\calH;v) \to K(\cG\ltimes\calH;lv) \to K(\cG;w)\ \ \text{and}$$
$$K(\calH)\to K(\cG\ltimes\calH)\to K(\cG).$$
\end{prop}

\begin{proof}[\bf Proof] 
Let us denote the category of 
acyclic complexes in $\cG\ltimes\calH$ 
by ${(\cG\ltimes\calH)}^q$. 
Since ${(\cG\ltimes\calH)}^q$ is 
closed under extensions in $\cG\ltimes\calH$, 
it naturally becomes a Waldhausen category 
and the association 
$x \mapsto [x \onto{\id_x} x]$ gives 
a equivalence 
between $\calH$ and ${(\cG\ltimes\calH)}^q$ 
as Waldhausen categories. 
On the other hand, 
there is an equivalence of Waldhausen categories 
$$\cG\ltimes\calH \isoto 
\E({(\cG\ltimes\calH)}^q, \cG\ltimes\calH,\cG),$$
$$x \mapsto 
\begin{pmatrix}
\begin{bmatrix}
\xymatrix{
x_1 \ar[d]_{\scriptstyle{\id_{x_1}}}\\
x_1
}
\end{bmatrix} &
\begin{matrix}
\xymatrix{
\overset{\scriptstyle{\id_{x_1}}}{\to} \ar@{}[d]\\
\underset{\scriptstyle{d^x}}{\to}\\
}
\end{matrix} &
\begin{bmatrix}
\xymatrix{
x_1 \ar[d]_{\scriptstyle{d^x}}\\
x_0
}
\end{bmatrix} &
\begin{matrix}
\xymatrix{
\to \ar@{}[d]\\
\to\\
}
\end{matrix} &
\begin{bmatrix}
\xymatrix{
0 \ar[d] \\
\Homo_0(x)
}
\end{bmatrix}
\end{pmatrix}\ \ \ .$$
where $\E({(\cG\ltimes\calH)}^q,\cG\ltimes\calH,\cG)$ is 
the exact category of admissible exact sequences 
$x \rinf y \rdef z$ in $\cG\ltimes\calH$ 
such that 
$x$ is in ${(\cG\ltimes\calH)}^q$ and $z$ is in $\cG$. 
Moreover $\E({(\cG\ltimes\calH)}^q,\cG\ltimes\calH,\cG)$ 
has the natural class of weak equivalences 
$lv$ induced from $\cG\ltimes\calH$. 
Hence by the additivity theorem in \cite[Proposition 1.3.2.]{Wal85}, 
we get the first fibration sequence. 
The second fibration sequence is given by taking 
$w=i_{\cG}$ and $v=i_{\calH}$ the class of all isomorphisms 
in $\cG$ and $\calH$ and by the equality 
$li_{\calH}=i_{\cG\ltimes\calH}$ in the first situation. 
\end{proof} 

\begin{df}[\bf Adroit systems]
\label{df:adroit system}
An {\bf adroit} (resp. a {\bf strongly adroit}) {\bf system} 
in an abelian category $\cA$ is a triple 
$\cX=(\cE_1,\cE_2,\cF)$ consisting of strict 
exact subcategories $\cE_1 \rinc \cE_2 \linc \cF$ of $\cA$ 
such that they satisfies the following conditions 
{\bf (Adr 1)}, {\bf (Adr 2)}, {\bf (Adr 3)} 
and {\bf (Adr 4)} 
(resp. {\bf (Adr 1)}, {\bf (Adr 2)}, {\bf (Adr 3)} 
and {\bf (Adr 5)}).\\
{\bf (Adr 1)} 
$\cF\ltimes\cE_1$ and $\cF\ltimes\cE_2$ are 
strict exact subcategories of $\Cub^1\cA$.\\
{\bf (Adr 2)} 
$\cE_1$ is closed under extensions in $\cE_2$.\\
{\bf (Adr 3)} 
Let 
$x \rinf y \rdef z$ 
be an admissible exact sequence in $\cA$. 
Assume that $y$ is isomorphic 
to an object in $\cE_1$ 
and $z$ is isomorphic 
to an object in either $\cE_1$ or $\cF$. 
Then $x$ is isomorphic 
to an object in $\cE_1$.\\ 
{\bf (Adr 4)} 
For any object $z$ in $\cE_2$, 
there is an object $y$ in $\cE_1$ 
and an admissible epimorphism $y \rdef z$ in $\cE_2$.\\ 
{\bf (Adr 5)} 
For any non-negative integer $m$ 
and any object $z$ in $\HOM([m],\cE_2)$, 
there are an object $y$ in $\HOM([m],\cE_1)$ 
and an admissible epimorphism $y \rdef z$ in $\HOM([m],\cE_2)$.\\ 
\end{df}

\begin{thm}
\label{semi-direct resolution theorem} 
Let $\cX=(\cE_1,\cE_2,\cF)$ be a triple of strict exact subcategories 
$\cE_1\rinc\cE_2\linc\cF$ of $\cA$ and $w$ a class of morphisms in $\cF$ 
such that $(\cF,w)$ is a Waldhausen exact category. 
Then\\
$\mathrm{(1)}$ 
If the triple $\cX$ is an adroit {\rm (}resp. a strongly adroit{\rm )} system, 
then the inclusion functor of opposite categories
${(\cF\ltimes \cE_1)}^{\op}\rinc{(\cF\ltimes \cE_2)}^{\op}$ 
{\rm (}resp. $({(\cF\ltimes \cE_1)}^{\op},tw^{\op})\rinc 
({(\cF\ltimes \cE_2)}^{\op},tw^{\op})${\rm )}
satisfies the resolution {\rm (}resp. strong resolution{\rm )} conditions in 
\ref{strong resol cond df}. 
In particular 
we have a homotopy equivalence 
on $K$-theory: 
$$K(\cF\ltimes \cE_1) \to K(\cF\ltimes \cE_2)$$ 
$${\rm (}{\text{resp.}}\ K(\cF\ltimes\cE_1;tw) 
\to K(\cF\ltimes\cE_2;tw){\rm )}.$$ 

\sn
$\mathrm{(2)}$ 
{\bf (Abstract weight declension theorem).} 
If the triple $\cX$ is a strongly adroit system, 
then the exact functor 
$\Homo_0:(\cF\ltimes\cE_1,tw) \to (\cF,w)$ 
induces a homotopy equivalence 
on $K$-theory:
$$
K(\Homo_0):K(\cF\ltimes\cE_1;tw) 
\to K(\cF;w).
$$

\sn
$\mathrm{(3)}$ 
If the triple $\cX$ is an adroit system, 
then for $i=1$, $2$, 
the exact functor $\theta_i:\cF\ltimes\cE_i \to \cF\times\cE_i$ 
which sends a morphism $x=[x_1\onto{f} x_0]$ to 
an ordered pair $(\Homo_0(x),x_1)$ 
gives a homotopy equivalence on $K$-theory:
$$K(\cF\ltimes\cE_i) \to K(\cF)\times K(\cE_i).$$
\end{thm}

\begin{proof}[\bf Proof] 
Proof of assertion $\mathrm{(1)}$: 
Let us fix a non-negative integer $m$. 
We will only prove that the inclusion functor 
${(\cF\ltimes\cE_1(m,tw))}^{\op} 
\rinc {(\cF\ltimes\cE_2(m,tw))}^{\op}$ 
satisfies the resolution conditions in \ref{resol cond df} when 
$\cX$ is a strongly adroit system. 
When $m=0$, this yields the case when $\cX$ is adroit. 
The condition \textbf{(Res 1)} follows from 
\ref{lem:closedunderextensions} and 
\ref{closed conditions 1}. 
We can easily check the condition \textbf{(Res 3)} from 
assumption {\bf (Adr 3)} and \ref{lem:isomclosed}. 
Next we check the condition \textbf{(Res 2)}. 
For each $x$ in $\cF\ltimes\cE_2(m,w)$, 
by assumption {\bf (Adr 5)}, 
we have an object $y_0$ 
in $\HOM([m],\cE_1)$ and 
an admissible epimorphism $y_0 \rdef x_0$. 
Then for each $0 \leqq i \leqq m$, 
we put 
$y_1(i):= \Ker(y_0(i) \rdef x_0(i) \rdef \Homo_0 x(i))$. 
We have the following commutative diagram:
$$\xymatrix{
y_1(i) \ar@{->>}[r] \ar@{>->}[d] 
& x_1(i) \ar@{>->}[d]\\ 
y_0(i) \ar@{->>}[r] \ar@{->>}[d] 
& x_0(i) \ar@{->>}[d]\\ 
\Homo_0 x(i) \ar@{=}[r] 
& \Homo_0 x(i)\ . 
}$$
By assumption {\bf (Adr 3)}, 
we notice that $y$ is 
in $\HOM([m],\cF\ltimes\cE_1)$. 
Since the morphism $y(i) \rdef x(i)$ is a quasi-isomorphism 
for each $0 \leqq i \leqq m$, 
we learn that $y$ is in $\cF\ltimes\cE_1(m,w)$. 
Therefore we get the result. 

\sn
Proof of assertion $\mathrm{(2)}$: 
We have the factorization
$$
K(\Homo_0): 
K(\cF\ltimes\cE_1;tw) 
\onto{\text{I}} 
K(\cF\ltimes\cE_2;tw) 
\onto{\text{II}} 
K(\cF;w)
$$
where the maps I and II are 
homotopy equivalences 
by 
assertion $\mathrm{(1)}$ 
and \ref{semi-direct homotopy equivalence} respectively.

\sn
Proof of assertion $\mathrm{(3)}$:
The inclusion functors $\cF\ltimes\cE_1 \rinc \cF\ltimes\cE_2$ 
and $\cE_1 \rinc \cE_2$ induce the commutative diagram below:
$${\footnotesize{\xymatrix{
K(\cF\ltimes\cE_1) \ar[r]^{\!\!K(\theta_1)} \ar[d]_{\wr} & 
K(\cF)\times K(\cE_1) \ar[d]^{\wr}\\
K(\cF\ltimes\cE_2) \ar[r]_{\!\!K(\theta_2)}^{\!\!\!\!\!\!\sim} & 
K(\cF)\times K(\cE_2).
}}}$$
Here the vertical maps are homotopy equivalences 
by $\mathrm{(1)}$ and the resolution theorem~\ref{resol thm} 
and 
the bottom map $K(\theta_2)$ is a homotopy equivalence 
by \ref{Split fibration theorem}. 
Therefore the map $K(\theta_1)$ is also a homotopy equivalence. 
\end{proof}

\section{Admissible cubes}
\label{sec:adm cubes}

In this section 
we define and study the notion of 
an admissible cube in an abelian category 
which is 
a categorical variant of the concept about 
regular sequences. 
We calculate the homologies of 
the total complexes of admissible cubes in 
\ref{prop:homology of admissible cubes} 
and as its applications, 
we give a characterization of admissible cubes in terms of 
their faces and total complexes as in \ref{cor:charof adm cube} 
and an inductive characterization of admissibility 
as in \ref{cor:admcriterion I}. 
Finally by utilizing the notion of admissibility, 
we extend semi-direct products to multi semi-direct products of 
exact categories as in \ref{df:mult semi-direct prod}. 
Let us start by organizing the general phraseologies of cubes. 
Let $\cA$ be an abelian category.

\begin{df}[\bf Restriction of cubes]
\label{df:restriction}
Let $U$, $V$ be a pair of disjoint subsets of $S$. 
We define the functor $i_U^V:\cP(U)\to \cP(S),\ A \mapsto A\cup V$ 
and it induces the natural transformation 
${(i_U^V)}^{\ast}:\Cub^S \to \Cub^U$. 
For any $S$-cube $x$ in a category $\cC$, 
we write $x|_U^V$ for ${(i_U^V)}^{\ast}x$ and it 
is called {\bf restriction of $x$} (to $U$ along $V$).
\end{df}

\begin{ex}[\bf Faces of cubes]
\label{ex:Faces of cubes}
For any $S$-cube $x$ in a category $\cC$ 
and $k\in S$, 
$x|_{S\ssm\{k\}}^{\{k\}}$, $x|_{S\ssm\{k\}}^{\emptyset}$ are called 
the {\bf backside $k$-face of $x$}, 
the {\bf frontside $k$-face of $x$} respectively. 
By a {\bf face} of $x$, 
we mean any backside or frontside $k$-face of $x$.
\end{ex}

\sn
Recall the definition of $\Homo_0^u$ and $\Homo_1^u$ for cubes from \ref{cube homology}.

\begin{lem}
\label{lem:compat hom and rest}
For any $S$-cube $x$ in an abelian category and 
any pair of disjoint subsets $U$ and $V$ and any element $u$ in $U$, 
we have 
\begin{equation}
\label{equ:compat hom and rest}
\Homo_p^u(x|_U^V)=\Homo_p^u(x)|_{U\ssm\{u\}}^V \ \ \ \text{for $p=0$, $1$.}
\end{equation}
\end{lem}

\begin{proof}[\bf Proof]
We will only prove the equation $\mathrm{(\ref{equ:compat hom and rest})}$ for $p=0$. 
For any subset $T$ of $U\ssm\{u\}$, 
the following equalities show the equation $\mathrm{(\ref{equ:compat hom and rest})}$ for $p=0$.
\begin{multline*} 
\Homo_0^u(x|_U^V)|_T  =  
\coker({(x|_U^V)}_{T\coprod\{u\}} \onto{d_{T\coprod\{u\}}^{x|_U^V,u}} {(x|_U^V)}_T)\\
 =  
\coker(x_{V\coprod T\coprod\{u\}} \onto{d^{x,u}_{V\coprod T\coprod\{u\}}} x_{V\coprod T})
 =  
\Homo_0^u(x)_{V\coprod T}
 =   
{(\Homo_0^u(x)|_{U\ssm\{u\}}^V)}_T.
\end{multline*}
\end{proof}

\begin{df}[\bf Degenerate cubes, non-degenerate cubes]
\label{df:degenerate cube}
Let $x$ be an $S$-cube in a category $\cC$.\\ 
$\mathrm{(1)}$ 
For $k\in S$, 
we say that $x$ is 
{\bf degenerate along the $k$-direction} 
if $d_{T\coprod\{k\}}^{x,k}$ is an isomorphism 
for any $T\in \cP(S\ssm\{k\})$.\\ 
$\mathrm{(2)}$ 
We say that 
$x$ is {\bf non-degenerate} if no boundary morphism of $x$ 
is an isomorphism. 
\end{df}

\begin{df}[\bf Total complexes]
\label{df:Tot comp}
For an $n$-cube $x$ in an additive category $\cB$, 
we associate the complex $\Tot x$, called 
the {\bf total complex} of $x$, defined as: 
$$\displaystyle{{(\Tot x)}_k:=\underset{\substack{T\in\cP((n]) \\ 
\# T=k}}{\bigoplus} x_T}.$$ 
The boundary morphisms 
$d_k^{\Tot x}:{(\Tot x)}_k \to {(\Tot x)}_{k-1}$ 
are defined by 
$$(-1)^{\overset{n}{\underset{t=j+1}{\sum}}\chi_T(t)}d_T^j:x_T 
\to x_{T\ssm\{j\}}$$
on its $x_T$ component to $x_{T\ssm\{j\}}$ component. 
Here $\chi_T$ is the characteristic function associated with $T$. 
(See \ref{rem:cubentrem}). 
For a general finite set $S$, 
let us fix a bijection 
$\alpha:(n]\isoto S$. 
Then we can consider any 
$S$-cubes to be $n$-cubes by $\alpha$. 
Therefore we can define the total complex of an $S$-cube $x$ 
which is denoted by $\Tot_{\alpha}x$ or simply $\Tot x$. 
Next 
moreover let us assume that $\cB$ is an abelian category. 
We say that a morphism $f:x\to y$ between $S$-cubes in $\cB$ 
is a {\bf quasi-isomorphism} if 
$\Tot f:\Tot x \to \Tot y$ is a quasi-isomorphism. 
We denote the class of quasi-isomorphisms 
in $\Cub^n\cB$ by $\tq_{\Cub^n\cB}$ or shortly $\tq$.
\end{df}

\begin{df}[\bf Spherical complex]
\label{df:spherical complex}
Let $n$ be an integer. 
We say that a complex $z$ 
in an abelian category $\cA$ is {\bf $n$-spherical} if 
$\Homo_k(z)=0$ for any $k\neq n$.
\end{df}

\sn
From now on, 
in this section, 
let us assume that $S$ is a finite set and 
let $x$ be an $S$-cube in an abelian category $\cA$. 

\begin{ex}[\bf Motivational example] 
\label{ex:mot ex} 
Let $\{f_s\}_{s\in S}$ be a family of elements 
in $A$ which forms a regular sequence in any order. 
Then for any $1\leq k\leq \# S$ and any 
distinct elements $s_1,\cdots,s_k$ in $S$, 
the boundary maps in the cube 
$\Homo_0^{s_1}(\cdots(\Homo_0^{s_k}(\Typ_A(\ff_S)))\cdots)$ 
are injections 
where $\Typ_A(\ff_S)$ is the typical cube associated with 
$\{f_s\}_{s\in S}$ (See \ref{Koszul cube def}). 
\end{ex}

\begin{df}[\bf Admissible cubes]
\label{df:adm cube}
If $\# S=1$, $x$ is said to be {\bf admissible} if 
its boundary morphism is a monomorphism. 
Inductively, for $\# S>1$, $x$ is called {\bf admissible} if 
its boundary morphisms are monomorphisms and if 
for any $k$ in $S$, 
$\Homo^k_0(x)$ is admissible. 
By convention, we say that 
$x$ is admissible if $S=\emptyset$. 
\end{df}

\begin{rem}
\label{rem:admissibility} 
The name \lq\lq admissible\rq\rq\ comes from \cite[p.331]{Wal85}. 
In \cite[p.323 1.1.2]{Wal85}, 
Waldhausen give a characterization of admissibility of squares. 
More precisely, 
for a square of monomorphisms in $\cA$ below
$$\footnotesize{\xymatrix{
a \ar@{>->}[r] \ar@{>->}[d] & b \ar@{>->}[d]\\
c \ar@{>->}[r] & d ,
}}$$
it is admissible if and only if the induced morphism 
$c\coprod_a b \to d$ is a monomorphism. 
\end{rem}

\begin{rem}
\label{rem:admissible check}
Assume that $\# S\geq 2$. 
Then 
$x$ is admissible if and only if 
for any 
distinct elements 
$t$, $s_1,\cdots,s_r$ in $S$ ($1\leq r\leq \#S -1$), 
$\Homo_1^t(x)$ and 
$\Homo_1^t(\Homo_0^{s_1}(\Homo_0^{s_2}(\cdots(\Homo_0^{s_r}(x))\cdots)))$ 
are trivial. 
\end{rem}

\begin{lemdf}
\label{lemdf:relative homology}
For a pair of disjoint subsets 
$U$ and $V\in\cP(S)$ such that  
$k:=\# U\geq 2$, let us assume that  
$x|^V_U$ is admissible. 
We denote the all distinct elements of $U$ by 
$i_1,\ldots,i_k$.  
Then we have the canonical isomorphism:
$$\Homo^{i_1}_0(\Homo^{i_2}_0(\cdots(\Homo_0^{i_k}(x))\cdots))_V\isoto 
\Homo^{i_{\sigma(1)}}_0(\Homo^{i_{\sigma(2)}}_0(\cdots(
\Homo_0^{i_{\sigma(k)}}(x))\cdots))_V$$
where $\sigma$ is a bijection on $U$. 
In this case 
we put 
$\Homo^U_0(x)_V:=
\Homo^{i_1}_0(\Homo^{i_2}_0(\cdots(\Homo^{i_k}_0(x))\cdots))_V$. 
We also put $\Homo_0^{\emptyset}(x):=x$. 
Notice that $\Homo_0^T(x)$ is an $S\ssm T$-cube for any $T\in \cP(S)$.
\end{lemdf}

\begin{proof}[\bf Proof] 
We may assume that $U=S$ and $V=\emptyset$ 
by replacing $x|_U^V$ with $x$. 
Since every bijection on $S$ is 
expressed in 
compositions of 
substitutions of two elements, 
we shall just check the assertion 
for any substitution of two elements $\sigma$. 
Since for a pair of distinct elements $i$, $j\in S$, 
$x$ is considered to be 
a $\{i,\ j\}$-cube 
in $\Cub^{S\ssm\{i,\ j\}}\cA$ 
by \ref{rem:disjointindex}, 
we shall assume that $S=\{i,\ j\}$. 
In this case, by $3\times 3$-lemma 
(See for example \cite[Exercise 3.2.1]{Wei94}), 
we learn that 
$\Homo^j_0(\Homo^i_0(x))_{\emptyset}$ and 
$\Homo^i_0(\Homo^j_0(x))_{\emptyset}$ 
are canonically isomorphic to 
the object $y$ in the diagram below. 
$$\footnotesize{\xymatrix{
x_{\{i,\ j \}} \ar@{>->}[r] \ar@{>->}[d] & x_{\{i\}} 
\ar@{>->}[d] \ar@{->>}[r] & 
\Homo_0^j(x)_{\{i\}} \ar@{>->}[d]\\
x_{\{j\}} \ar@{>->}[r] \ar@{->>}[d] & 
x_{\emptyset} \ar@{->>}[r] \ar@{->>}[d] & 
\Homo_0^j(x)_{\emptyset} \ar@{->>}[d]\\
\Homo^i_0(x)_{\{j\}} \ar@{>->}[r] & 
\Homo_0^i(x)_{\emptyset} \ar@{->>}[r] & 
y  .
}}$$
\end{proof}

\begin{lem}
\label{lem:face of admissible cubes}
Assume that $x$ is admissible. 
Then, for any pair of disjoint subsets $U$ and $V$ of $S$, 
$x|_U^V$ is admissible. 
In particular, 
all faces of $x$ are also admissible when $\# S\geq 1$.
\end{lem}

\begin{proof}[\bf Proof]
If $\# U\leq 1$, then the assertion is trivial. 
We assume that $\# U \geq 2$. 
For any distinct elements $u$, $s_1,\cdots,s_r$ of $U$, 
since $\Homo_1^u(x)$ and 
$\Homo_1^u(\Homo^{s_1}_0(\Homo_0^{s_2}(\cdots(\Homo_0^{s_r}(x))\cdots)))$ 
are trivial by assumption and \ref{rem:admissible check}, 
$\Homo_1^u(x|_U^V)$ and 
$\Homo_1^u(\Homo^{s_1}_0(\Homo_0^{s_2}(\cdots(\Homo_0^{s_r}(x|_U^V))\cdots)))$ 
are also trivial by the equalities
$$\Homo_1^u(x|_U^V)=\Homo_1^u(x)|^V_{U\ssm\{u\}},$$
$$\Homo_1^u(\Homo^{s_1}_0(\Homo_0^{s_2}(\cdots(\Homo_0^{s_r}(x|_U^V))\cdots)))
=\Homo_1^u(\Homo^{s_1}_0(\Homo_0^{s_2}(\cdots(\Homo_0^{s_r}(x))\cdots)))
|^V_{U\ssm\{u,s_1,\cdots,s_r\}}$$
which come from \ref{lem:compat hom and rest}. 
Hence $x|_U^V$ is admissible.
\end{proof}

\begin{prop}
\label{prop:homology of admissible cubes}
Let us assume that 
$S$ is a non-empty set and 
all faces of $x$ are admissible. 
Then\\
$\mathrm{(1)}$ 
For any element 
$k$ in $S$, 
we have the following isomorphisms
$$
\omega_{k,S,x}^p:  
\Homo_p(\Tot x)\isoto 
\begin{cases}
\Homo_p^k(\Homo_0^{S\ssm\{k\}}(x)) & \text{for $p=0$, $1$}\\
0 & \text{otherwise}
\end{cases}
$$
which is functorial in the following sense.

\sn
For any $S$-cube $y$ in $\cA$ such that all faces are admissible 
and for any morphism $f:x\to y$ of $S$-cubes, 
the following diagram is commutative for $p=0$, $1$.
$${\footnotesize{\xymatrix{
\Homo_p\Tot x \ar[r]^{\Homo_p\Tot (f)} \ar[d]^{\wr}_{\omega^p_{k,S,x}} & 
\Homo_p\Tot y \ar[d]_{\wr}^{\omega^p_{k,S,y}}\\
\Homo_p^k\Homo_0^{S\ssm\{k\}}(x) \ar[r]_{\Homo_p^k\Homo_0^{S\ssm\{k\}}(f)} &
\Homo_p^k\Homo_0^{S\ssm\{k\}}(y).
}}}$$

\sn
$\mathrm{(2)}$ 
In particular, if 
$x$ is admissible, then we have the isomorphisms. 
$$
\eta_{S,x}^p:  
\Homo_p(\Tot x)\isoto 
\begin{cases}
\Homo_0^{S}(x) & \text{for $p=0$}\\
0 & \text{otherwise}
\end{cases}
.$$
\end{prop}

\begin{proof}[\bf Proof]
First we prove that 
assertion $\mathrm{(1)}$ implies assertion $\mathrm{(2)}$. 
Assume $x$ is admissible. 
Then all faces of $x$ and $\Homo_0^{S\ssm\{k\}}(x)$ 
for all $k\in S$ 
are admissible by \ref{lem:face of admissible cubes} 
and the definition of admissibility. 
In particular $\Homo_1^k\Homo_0^{S\ssm\{k\}}(x)$ is trivial. 
Hence we obtain the result by $\mathrm{(1)}$.

\sn
Next we will prove assertion $\mathrm{(1)}$. 
We proceed by induction on the cardinality of $S$. 
Let us assume that 
the assertion is true for any $S\ssm\{k\}$-cubes in $\cA$ 
which satisfy the assumption. 
For simplicity, we put $T=S\ssm \{k\}$ and 
$y_1=x|_T^{\{k\}}$ and $y_0=x|_T^{\emptyset}$ and 
$d:=d_T^{x,k}$. 
We regard $x$ as the $\{k\}$-cube 
$x=[y_1 \onto{d} y_0]$ 
in $\Cub^{T}\cA$. 
Since we have the isomorphisms 
$\Tot\Cone d\isoto \Cone \Tot d\isoto \Tot x$, 
there is a distinguished triangle
\begin{equation}
\label{equ:dist triangle}
\Tot y_1 \onto{\Tot d} \Tot y_0 \to \Tot x \onto{+1}
\end{equation}
in the homotopy category of bounded complexes on $\cA$. 
Here $\Homo_p\Tot y_q$ is trivial for any $p\geq 1$ and $q=0$, $1$ 
by the inductive assumption and assertion $\mathrm{(2)}$. 
Therefore 
the long exact sequence induced from the distinguished triangle 
$\mathrm{(\ref{equ:dist triangle})}$ 
shows that $\Homo_p\Tot x$ is trivial for $p\geq 2$ 
and  
yields the top exact sequence in the commutative diagram below.
$${\footnotesize{\xymatrix{
0 \ar[r] & \Homo_1\Tot x  
\ar[r] \ar@{-->}[dd]_{\omega^1_{k,S,x}}^{\wr}
& \Homo_0\Tot y_1 \ar[r]^{\Homo_0\Tot d} \ar[d]^{\wr}_{\eta^0_{T,y_1}}
& \Homo_0\Tot y_0 \ar[r] \ar[d]_{\wr}^{\eta^0_{T,y_0}}
& \Homo_0\Tot x \ar[r] \ar@{-->}[dd]^{\omega^0_{k,S,x}}_{\wr} & 0\\
& & \Homo_0^T(y_1)  \ar@{=}[d]&
\Homo_0^T(y_0) \ar@{=}[d]\\
0 \ar[r] & 
\Homo_1^k\Homo_0^T(x) \ar[r] 
& {\Homo_0^T(x)}_{\{k\}} \ar[r]_{d_{\{k\}}^{\Homo_0^T(x),k}} 
& {\Homo_0^T(x)}_{\emptyset} \ar[r] 
& \Homo_0^k\Homo_0^T(x) \ar[r] & 0.
}}}$$
Then there are the isomorphisms 
$\omega^p_{k,S,x}:\Homo_p\Tot x\isoto \Homo_p^k\Homo_0^T(x)$ 
for $p=0$, $1$ 
which make the diagram above commutative. 
By construction, $\omega^p_{k,S,x}$ is functorial. 
\end{proof}

\begin{cor}
\label{cor:homology of admissible cubes}
For a pair of disjoint finite subsets $U$, $V$ of $S$, 
let us assume that $x|_U^V$ is admissible. 
Then we have the isomorphisms below:
$$\Homo_p(\Tot(x|_U^V))\isoto
\begin{cases}
\Homo_0^U(x)_V & \text{for $p=0$}\\
0 & \text{otherwise}
\end{cases}
.$$
\end{cor}

\begin{proof}[\bf Proof]
If $U=\emptyset$, 
we have the equality 
$x|_U^V=x_V=\Homo_0^{\emptyset}(x)_V$. 
Therefore the assertion is obvious. 
If $U\neq \emptyset$
applying \ref{prop:homology of admissible cubes} 
to $x|^V_U$ and noticing that 
the equality ${\Homo_0^U(x|_U^V)}_{\emptyset}=\Homo_0^U(x)_V$, 
we get the result.
\end{proof}

\begin{cor}
\label{cor:charof adm cube}
The following conditions are equivalent.\\ 
$\mathrm{(1)}$ 
$x$ is admissible.\\ 
$\mathrm{(2)}$ 
$\Tot x$ is $0$-spherical and all faces of $x$ are admissible. 
\end{cor}

\begin{proof}[\bf Proof]
Condition $\mathrm{(1)}$ implies 
condition $\mathrm{(2)}$ by 
\ref{lem:face of admissible cubes} 
and 
\ref{prop:homology of admissible cubes} $\mathrm{(2)}$. 
Conversely let us assume that $x$ satisfies condition $\mathrm{(2)}$. 
We may assume $\# S\geq 1$. 
We prove that for any 
disjoint pair $V$ and $W\in\cP(S)\ssm\{S\}$ and $k\in W$, 
the boundary morphism 
$d_W^{\Homo_0^V(x),k}:\Homo_0^V(x)_W \to \Homo_0^V(x)_{W\ssm\{k\}}$ 
is a monomorphism. 
If $\#(S\ssm V)\geqq 2$, 
set $x'$ equal to $x|_{V\coprod\{k\}}^{W\ssm\{k\}}$
and then we can identify the boundary morphism above with 
the following morphism
$$d_{\{k\}}^{\Homo_0(x'),k}:
\Homo_0(x')_{\{k\}} \to 
\Homo_0(x')_{\emptyset}.$$
Therefore the assertion follows from admissibility of $x'$. 
If $S\ssm V$ is the singleton $\{k\}$, then 
$W=\{k\}$ and 
by \ref{prop:homology of admissible cubes}, 
we have the isomorphisms 
$\Homo_1^k(\Homo_0^{S\ssm\{k\}}(x))\isoto\Homo_1(\Tot x)=0$. 
This means that we get the desired result.
\end{proof}

\begin{cor}[\bf Inductive characterization of admissibility]
\label{cor:admcriterion I}
Assume that $S$ is a non-empty set. 
Then the following conditions are equivalent.\\ 
$\mathrm{(1)}$ 
$x$ is admissible.\\ 
$\mathrm{(2)}$ 
For some $s\in S$, 
$x$ satisfies the following three conditions.\\ 
$\mathrm{(i)}$ 
$x|^{\{s\}}_{S\ssm\{s\}}$, $x|^{\emptyset}_{S\ssm\{s\}}$ are 
admissible.\\ 
$\mathrm{(ii)}$ 
$d^{x,s}_{T\coprod\{s\}}$ is a monomorphism for any 
$T\in\cP(S\ssm\{s\})$.\\ 
$\mathrm{(iii)}$ 
$\Homo_0^s(x)$ is admissible.\\ 
$\mathrm{(3)}$ 
For any $s\in S$, 
$x$ satisfies the three conditions 
$\mathrm{(i)}$, $\mathrm{(ii)}$ and $\mathrm{(iii)}$ 
in $\mathrm{(2)}$.
\end{cor}

\begin{proof}[\bf Proof] 
We can easily check that condition 
$\mathrm{(1)}$ implies condition $\mathrm{(3)}$ 
and condition $\mathrm{(3)}$ implies condition $\mathrm{(2)}$. 
We need only prove that condition $\mathrm{(2)}$ 
implies condition $\mathrm{(1)}$. 
We prove this assertion by induction on $\# S$. 
If $\# S=1$, the assertion is trivial. 
For $\#S>1$, we will prove that\\ 
$\mathrm{(a)}$ 
all faces of $x$ are admissible, and\\ 
$\mathrm{(b)}$ 
$\Tot x$ is $0$-spherical.

\sn
Proof for assertion $\mathrm{(a)}$: 
We prove that 
for any $k\in S$, 
the faces $x|_{S\ssm\{k\}}^{\{k\}}$ and 
$x|_{S\ssm\{k\}}^{\emptyset}$ are admissible. 
If $k=s$, it is just condition $\mathrm{(i)}$. 
If $k\neq s$, then they satisfy 
conditions $\mathrm{(i)}$, $\mathrm{(ii)}$ and $\mathrm{(iii)}$ 
and therefore by the inductive hypothesis, 
they are admissible. 

\sn
Proof for assertion $\mathrm{(b)}$: 
Fix an element $t\in S\ssm\{s\}$. 
Since we have the isomorphism 
$$\Homo^{S\ssm\{t\}}_0(x)\isoto\Homo_0^{S\ssm\{s,\ t\}}
(\Homo_0^s(x)),$$
we learn that 
$\Homo_0^{S\ssm\{t\}}(x)$ is admissible by condition $\mathrm{(iii)}$. 
In particular, 
$\Homo_1^t(\Homo_0^{S\ssm\{t\}}(x))$ is trivial. 
On the other hand, 
by \ref{prop:homology of admissible cubes}, 
we have the isomorphism
$$\Homo_p(\Tot x)\isoto
\begin{cases}
\Homo_p^t(\Homo^{S\ssm\{t\}}_0(x)) & 
\text{for $p=0$, $1$}\\
0 & 
\text{otherwise}
\end{cases}. $$
Therefore we notice that $\Tot x$ is $0$-spherical.

\sn
Hence by \ref{cor:charof adm cube}, 
$x$ is admissible. 
\end{proof}

\begin{cor}
\label{cor:admcriterion II}
For a subset $T\subset S$, 
if the following two conditions are verified, 
then $x$ is admissible.\\ 
$\mathrm{(1)}$ 
$x$ is degenerate along the $k$-direction for any $k\in T$.\\ 
$\mathrm{(2)}$ 
$x|_{S\ssm T}^{\emptyset}$ is admissible.
\end{cor}

\begin{proof}[\bf Proof]
Since for any element in $t\in T$, 
we have the equality
$${(x|_{S\ssm(T\ssm\{t\})}^{\emptyset})|}_{S\ssm T}^{\emptyset}
=x|_{S\ssm T}^{\emptyset},$$
by the induction of $\# T$, 
we shall assume that $T$ is the singleton $T=\{k\}$. 
In this case, 
$x$ satisfies conditions 
$\mathrm{(i)}$, $\mathrm{(ii)}$ and $\mathrm{(iii)}$ 
in \ref{cor:admcriterion I} $\mathrm{(2)}$ for $s=k$. 
Therefore $x$ is admissible.
\end{proof}

\begin{df}[\bf Multi semi-direct products]
\label{df:mult semi-direct prod}
Let $\fF=\{\cF_T\}_{T\in\cP(S)}$ be 
a family of full subcategories of $\cA$. 
Then\\ 
$\mathrm{(1)}$ 
We define 
$\ltimes \fF=\underset{T\in\cP(S)}{\ltimes} \cF_T$ 
the {\bf multi semi-direct product 
of the family $\fF$} 
as follows. 
$\underset{T\in\cP(S)}{\ltimes} \cF_T$ 
is the full subcategory of $\Cub^S(\cF_{\emptyset})$ 
consisting of those cubes $x$ 
such that $x$ is admissible and each vertex of 
$\Homo^T_0(x)$ is in $\cF_T$ for any $T\in\cP(S)$. 
If $S$ is a singleton, 
then $\underset{T\in\cP(S)}{\ltimes} \cF_T$ is 
just the semi-direct product $\cF_S\ltimes \cF_{\emptyset}$ in 
\ref{semi-direct product nt}.\\
$\mathrm{(2)}$ 
Suppose in addition 
that $\cF_S$ has a class of weak equivalences 
$w=w(\cF_S)$. 
Then a morphism $f:x\to y$ in $\ltimes \fF$ is 
said to be a {\bf total weak equivalences} if 
$\Homo^S_0(f)$ is in $w$. 
We denote the class of total weak equivalences in $\ltimes \fF$ by 
$\displaystyle{tw(\underset{T\in \cP(S)}{\ltimes} \cF_T)}$ or simply $tw$.
\end{df}

\begin{prop}
\label{prop:rel multi and mono semi-direct prod}
Let $\fF=\{\cF_T\}_{T\in\cP(S)}$ be a family of 
full subcategories of $\cA$. 
Then\\ 
$\mathrm{(1)}$ 
For each $s\in S$, we have the equality
$$\underset{T\in\cP(S)}{\ltimes}\cF_T=
\left ( \underset{T\in \cP(S\ssm\{s\})}{\ltimes} 
\cF_{T\coprod\{s\}} \right) \ltimes 
\left ( \underset{T\in\cP(S\ssm\{s\})}{\ltimes} \cF_T \right ) .$$
$\mathrm{(2)}$ 
In the equality in $\mathrm{(1)}$, 
the class of quasi-isomorphisms in 
$\displaystyle{\underset{T\in\cP(S)}{\ltimes}\cF_T}$ is 
equal to the class of weak equivalences induced from 
the class of quasi-isomorphisms in 
$\displaystyle{\underset{T\in\cP(S\ssm\{s\})}{\ltimes} 
\cF_{T\coprod\{s\}}}$ 
as in \ref{semi-direct we nt}. 
Namely we have the equality 
$$tw\left (\underset{T\in\cP(S)}{\ltimes}\cF_T \right )=
t\left (tw \left (\underset{T\in\cP(S\ssm\{s\})}{\ltimes} \cF_{T\coprod \{s\}} \right )\right ).$$
$\mathrm{(3)}$ 
In the situation $\mathrm{(2)}$, 
we have the equality
$${\left ( \underset{T\in\cP(S)}{\ltimes}\cF_T
\right )}^{tw}
={\left ( 
\underset{T\in\cP(S\ssm\{s\})}{\ltimes} \cF_{T\coprod \{s\}}
\right )}^{tw} 
\ltimes \left ( 
\underset{T\in\cP(S\ssm\{s\})}{\ltimes} \cF_T
\right ) .$$ 
\end{prop}

\begin{proof}[\bf Proof] 
Proof of assertion $\mathrm{(1)}$: 
Let us fix an element $s\in S$. 
For simplicity, 
we put 
$$\displaystyle{\cG=
\left ( \underset{T\in \cP(S\ssm\{s\})}{\ltimes} 
\cF_{T\coprod\{s\}} \right) \ltimes 
\left ( \underset{T\in\cP(S\ssm\{s\})}{\ltimes} \cF_T \right )}.$$ 
Let $x$ be an object in $\ltimes \cF$. 
To prove that $x$ is in $\cG$, 
we need to check the following two assertions.\\ 
$\mathrm{(a)}$ 
$x|_{S\ssm\{s\}}^{\{s\}}$, $x|_{S\ssm\{s\}}^{\emptyset}$ are 
in $\underset{T\in\cP(S\ssm \{s\})}{\ltimes}\cF_T$.\\ 
$\mathrm{(b)}$ 
$\Homo_0^{s}(x)$ is in 
$\underset{T\in\cP(S\ssm \{s\})}{\ltimes}\cF_{T\coprod\{s\}}$.\\ 
We put $W=\emptyset$ or $W=\{s\}$. 
First let us notice that 
$x|_{S\ssm\{s\}}^W$ and 
$\Homo_0^{s}(x)$ are admissible 
by admissibility of $x$. 
For each $T\in \cP(S\ssm\{s\})$ and $V\in\cP(S\ssm (T\coprod\{s\}))$, 
we have the equalities 
\begin{equation}
\label{equ:rest equ I}
\Homo_0^T(\Homo_0^s(x))_V=\Homo_0^{T\coprod\{s\}}(x)_V \text{ and}  
\end{equation}
\begin{equation}
\label{equ:rest equ II}
\Homo_0^T(x|_{S\ssm\{s\}}^{W})_V=\Homo_0^T(x)_{V\coprod W}. 
\end{equation}
Therefore both objects above are in $\cF_{T\coprod W}$. 
Hence we get assertions (a) and (b) and 
we learn that $x$ is in $\cG$. 
Conversely next let $x$ be an object in $\cG$. 
Since $x$ satisfies condition $\mathrm{(i)}$, 
$\mathrm{(ii)}$ and $\mathrm{(iii)}$ in by \ref{cor:admcriterion I} (2), 
$x$ is admissible. 
Moreover by the equalities (\ref{equ:rest equ I}) 
and (\ref{equ:rest equ II}) above, 
we learn that $\Homo_0^T(x)_V$ is in $\cF_T$ for any disjoint pair 
of subsets $T$, $V\in \cP(S)$. 
Therefore $x$ is in $\ltimes \fF$. 

\sn
Proof of assertion $\mathrm{(2)}$: 
By \ref{cor:homology of admissible cubes}, 
a morphism $f:x\to y$ in 
$\displaystyle{\underset{T\in\cP(S)}{\ltimes}\cF_T}$ 
is a quasi-isomorphism if and only if 
$\Homo_0(\Tot f):\Homo_0(\Tot x) \to \Homo_0(\Tot y)$ 
is an isomorphism. 
Since $\Homo_0(\Tot z)=\Homo_0(\Tot \Homo_0^s(z))$ 
($z=x$, $y$), 
this condition is equivalent to 
the assertion that 
the induced morphism 
$\Homo_0^s(f):\Homo^s_0(x) \to \Homo_0^s(y)$ is a 
quasi-isomorphism. 
Hence we get the result. 

\sn
Proof of assertion $\mathrm{(3)}$: 
By virtue of the equality in $\mathrm{(1)}$, 
we may assume that $S$ is a singleton. Namely $\# S=1$. 
For simplicity we put $\cE=\cF_{\emptyset}$ and $\cF=\cF_{S}$. 
For any object $x$ in $\cF\ltimes \cE$, 
the canonical morphism $0 \to x$ is in $tw$ if and only if 
the canonical morphism $0 \to \Homo_0(x)$ is in $w$ and 
the last assertion is equivalent to $\Homo_0(x)$ being in $\cF^w$. 
Hence we get the desired result. 
\end{proof}

\begin{cor}
\label{cor:mult semi is exact}
Let $\fE=\{\cE_T\}_{T\in\cP(S)}$ and 
$\fF=\{\cF_T\}_{T\in\cP(S)}$ be 
families of subcategories of 
$\cA$ such that for each $T\in\cP(S)$, 
$\cF_T\rinc \cE_T$ are strict exact subcategories of $\cA$ and 
the inclusion functor 
$\cE_T\rinc \cA$ is closed under extensions. 
Then\\ 
$\mathrm{(1)}$ 
$\ltimes \fE$ is closed under extensions in $\Cub^S\cA$. 
In particular $\ltimes \fE$ naturally becomes an exact category.\\ 
$\mathrm{(2)}$ 
If $\cF_T\rinc \cE_T$ is closed under extensions 
{\rm (}resp. taking admissible sub- and quotient objects, 
taking kernels of admissible epimorphisms, taking finite direct sum{\rm )} 
for any $T\in\cP(S)$, 
then the inclusion functor $\ltimes\fF \rinc \ltimes \fE$ 
is also closed under extensions 
{\rm (}resp. taking admissible sub- and quotient objects, 
taking kernels of admissible epimorphisms, taking finite direct sum{\rm )}.\\ 
$\mathrm{(3)}$ 
Assume that the following two conditions hold.\\ 
$\mathrm{(i)}$ 
For any pair of subsets $T\subset U$ in $S$, 
$\cE_U$ is full subcategory of $\cE_T$.\\ 
$\mathrm{(ii)}$ 
For any subset $T$ in $S$, 
every object in $\cF_T$ is a projective object in $\cE_T$.\\ 
Then every object in $\ltimes\fF$ is a projective object in $\ltimes \fE$. 
\end{cor}

\begin{proof}[\bf Proof]
Utilizing the inductive description of $\ltimes\fE$ and 
$\ltimes \fF$ as in \ref{prop:rel multi and mono semi-direct prod} 
$\mathrm{(1)}$, 
we get the results by induction and \ref{closed conditions 1}, 
\ref{closed conditions 2}, 
\ref{rem:closed condition} and 
\ref{prop:projobjinsemidirectprod}. 
\end{proof}

\begin{rem}
\label{rem:exactness}
Let $\fF=\{\cF_T\}_{T\in\cP(S)}$ be 
a family of strict exact subcategories of $\cA$. 
Assume that for any disjoint decomposition $S=U\coprod V$, 
$\displaystyle{\underset{T\in\cP(U)}{\ltimes}\cF_{T\coprod V}}$ 
is a strict exact subcategory of $\Cub^U(\cA)$.\\
$\mathrm{(1)}$ 
Since boundary morphisms of admissible cubes are monomorphisms, 
for each $s\in S$, the functor 
$$\Homo^s_0:\underset{T\in \cP(S)}{\ltimes} \cF_T \to 
\underset{T\in\cP(S\ssm\{s\})}{\ltimes} \cF_{T\coprod\{s\}}$$
is exact. 
Moreover by induction on the number of elements, we 
learn that for any $W\in\cP(S)$, 
$$\Homo^W_0:
\underset{U\in \cP(S)}{\ltimes} \cF_U \to 
\underset{T\in\cP(S\ssm W)}{\ltimes} \cF_{W\coprod T}
$$
is also an exact functor.\\
$\mathrm{(2)}$ 
In particular, since for the functors 
$\Homo_0^S$, $\Homo_0\Tot:\ltimes \fF \to \cF_S$, 
we have the isomorphism $\Homo_0^S\isoto\Homo_0\Tot$ 
by \ref{prop:homology of admissible cubes} $\mathrm{(2)}$, 
$\Homo_0\Tot$ is also an exact functor from $\ltimes \fF$ to $\cF_S$. 
\end{rem}

\section{Koszul cubes}
\label{sec:Koszul cubes}

In this section, 
we define 
Koszul cubes in \ref{df:Koszul cube df} and 
relate them to semi-direct products of 
exact categories of pure weight modules 
as in \ref{cor:comp df of Koscube}. 
The pivot of the theory is 
Theorem~\ref{thm:TotstrictKos is 0sph} 
which says that 
the total complexes associated with free non-degenerate Koszul 
cubes are $0$-spherical. 
Let us commence defining and recalling elementary facts about 
regular sequences. 
Let $R$ (resp., $A$) be 
a commutative ring with $1$ 
(resp., commutative noetherian ring with $1$).

\sn
By an {\bf $A$-sequence} we mean an 
$A$-regular sequence $f_1,\cdots,f_q$ such that 
any permutation of the $f_j$ 
is also an $A$-regular sequence.

\begin{ex}
\label{ex:A-sequence}
We enumerate fundamental properties of $A$-sequences from 
\cite[\S 16]{Mat86}.\\ 
$\mathrm{(1)}$ 
For any $A$-regular sequence $f_1,\cdots,f_q$ and a 
prime ideal $\pp$ in $A$ such that $f_i \in \pp$ 
for any $i\in (q]$, 
$f_1,\cdots,f_q$ is an $A_{\pp}$-regular sequence in $A_{\pp}$. 
In particular if $f_1,\cdots,f_q$ is an $A$-sequence, 
then it is also an $A_{\pp}$-sequence in $A_{\pp}$.\\ 
$\mathrm{(2)}$ 
For any $A$-regular sequence $f_1,\cdots,f_q$ and 
positive integers $\mu_1,\cdots,\mu_q$, 
$f_1^{\mu_1},\cdots,f_q^{\mu_q}$ is again an $A$-regular sequence. 
In particular if $f_1,\cdots,f_q$ is an $A$-sequence, 
then $f_1^{\mu_1},\cdots,f_q^{\mu_q}$ is also an $A$-sequence.\\ 
$\mathrm{(3)}$ 
Any $A$-regular sequence 
contained in the Jacobson radical of $A$ 
is automatically an $A$-sequence.
\end{ex}

\sn
The following lemma might be well-known. 
But I do not know a reference and we give a proof in minute detail. 

\begin{lem}
\label{lem:A-seq lem}
Let $R$ be a commutative ring with unit and 
$f_1,\cdots,f_n,\ g_1,\cdots,g_n$ elements in $R$. 
We put $h_i=f_ig_i$ for each $1\leq i\leq n$ and 
assume that each $f_i$ is not a unit and 
the sequence $h_1,\cdots,h_n$ is an $R$-sequence. 
Then $f_1,\cdots,f_n$ is also an $R$-sequence. 
\end{lem}

\begin{proof}[\bf Proof]
If $n=1$, $f_1$ is actually a non zero divisor. 
For $n>1$, 
by induction on $n$, 
we shall only check that 
if for some elements $x$ and $y_i^{(0)}$ in $R$ $(1\leq i\leq n)$,
we have the equality
\begin{equation}
\label{equ:regseq1}
\displaystyle{f_nx=\sum_{i=1}^{n-1}f_iy_i^{(0)},}
\end{equation}
then we have the equality
$\displaystyle{x=\sum_{i=1}^{n-1}f_iz_i}$ 
for some elements $z_i$ in $R$. 
Multiplying the equality (\ref{equ:regseq1}) 
by $\displaystyle{g=\prod^n_{i=1} g_i}$, 
we get the equality
$$h_n\left ( \prod_{i=1}^{n-1}g_i  \right ) x=\sum_{i=1}^{n-1}h_i 
\left ( \frac{g}{g_i} y_i^{(0)}\right ).$$
Since the sequence $h_1$,$\cdots$,$h_n$ is an $R$-sequence, 
there 
are elements $y_1^{(1)},\cdots,y_{n-1}^{(1)}$ in $R$ such that 
$$\displaystyle{
\left (\prod_{i=1}^{n-1}g_i\right )x=\sum_{i=1}^{n-1}h_iy_i^{(1)}.
}$$
Now let us define the polynomials $\varphi_{x,i}(Y_1,\cdots,Y_{n-1})$ 
in $R[Y_1,\cdots,Y_{n-1}]$ $(1\leq i\leq n)$ by the formula 
$$\tiny{\displaystyle{\varphi_{x,i}(Y_1,Y_2,\cdots,Y_{n-1})=
\begin{cases}
\displaystyle{g_1(x-f_1Y_1)} & \text{if $i=1$ and $n=2$}\\
\displaystyle{g_1 \left \{ \left (\prod_{j=2}^{n-1}g_j \right ) 
x-f_1Y_1 \right \}-
\sum_{j=2}^{n-1} h_jY_j } 
& \text{if $i=1$ and $n\geq 3$}\\
\displaystyle{g_i\left \{ \left (\prod_{j=i+1}^{n-1}g_j \right ) 
x-f_iY_i \right \}-
\left ( \sum^{i-1}_{j=1}f_jY_j +\sum_{j=i+1}^{n-1} h_jY_j\right )} & 
\text{if $2\leq i\leq n-2$ and $n\geq 4$}\\
\displaystyle{g_{n-1}(x-f_{n-1}Y_{n-1})-\sum^{n-2}_{j=1}f_jY_j} & 
\text{if $i=n-1$  and $n\geq 3$}\\
\displaystyle{x-\sum_{j=1}^{n-1}f_jY_j} & \text{if $i=n$}
\end{cases}
}}.$$

\begin{claim}
\label{claim:solutionstep}
For $1\leq i\leq n-1$, 
if $\varphi_{x,i}=0$ has a solution in $R$, 
then the equation $\varphi_{x,i+1}=0$ also has a solution in $R$.
\end{claim}

\begin{proof}[\bf Proof of claim]
Let a system 
$Y_j=y_j$ $(1\leq j\leq n-1)$ be a solution of $\varphi_{x,i}=0$. 
If $g_i$ is a unit, 
then the system
$$Y_j= 
\begin{cases}
g_i^{-1}y_j & \text{if $j\neq i$}\\
y_i & \text{if $j=i$}
\end{cases}
$$ 
is a solution of $\varphi_{x,i+1}=0$. 
If $g_i$ is not a unit, 
then by inductive hypothesis, 
the sequence
$$\begin{cases}
g_1 & \text{(if $i=1$ and $n=2$)}\\
h_2,\cdots,h_{n-1},\ g_1 & \text{(if $i=1$ and $n\geq 3$)}\\
f_1,\ f_2,\cdots,f_{i-1},\ h_{i+1},\cdots,h_{n-1},\ g_i & 
\text{(if $2\leq i\leq n-2$ and $n\geq 4$)}\\
f_1,\cdots,f_{n-2},\ g_{n-1} & \text{(if $i=n-1$ and $n\geq 3$)}
\end{cases}
$$
is an $R$-sequence. 
Therefore we have elements $z_j$ in $R$ $(1\leq j \leq n-1,\ j\neq i)$ 
such that the system
$$Y_j=
\begin{cases}
z_j & \text{if $j\neq i$}\\
y_i & \text{if $j=i$}
\end{cases}$$
is a solution of $\varphi_{x,i+1}=0$. 
\end{proof}

\sn
Since the system $Y_j=y_j^{(1)}$ $(1\leq j \leq n-1)$ is a 
solution of $\varphi_{x,1}=0$, 
by induction on $i$, finally we have a solution of $\varphi_{x,n}=0$ 
and this is just the desired result.
\end{proof}

\sn
The following assertion is a very basic fact 
about regular sequences 
and useful to handle Koszul cubes. 

\begin{lem}
\label{lem:localize by reg seq} 
For a non-zero divisor $f$ in $R$, 
an $R$-sequence $g$, $h$ in $R$, 
we have the equalities 
$$R^{\times}_f\cap R=\{r\in R;\text{$f^m=ru$ for some $m\in\bbN$, 
$u\in R$}\}\ \ \text{and}$$
$$R^{\times}_g\cap R^{\times}_h=R^{\times}$$
in the total quotient ring of $R$. 
\end{lem}

\begin{proof}[\bf Proof] 
We will give a proof of the second equality above. 
An element $u\in R_g \cap R_h$ is of the form 
$\displaystyle{u=\frac{x}{g^n}=\frac{y}{h^m}}$ 
for some positive integers $n$ and $m$. 
Then we have the equality $h^mx=g^ny$. 
Since the sequence $g^n$, $h^m$ is 
an $R$-sequence by \ref{ex:A-sequence} (2), 
there is an element $z$ in $R$ such that 
$x=g^nz$ and therefore 
$\displaystyle{u=\frac{z}{1}}$.
Hence we have $R_g \cap R_h=R$ in the total quotient ring of $R$ 
and it implies the desired equality 
$R_g^{\times}\cap R_h^{\times}={(R_g\cap R_h)}^{\times}=R^{\times}$. 
\end{proof}

\sn
Next we will make ready for the general jargons of cubes in the category of 
Modules over a commutative ring with unit $R$. 

\begin{df}[\bf Free, projective, finitely generated cubes]
\label{df:free, projective, finitely generated cubes} 
We say that a cube $x$ in the category of $R$-Modules is 
{\bf free} (resp. {\bf projective}, {\bf finitely generated}) 
if all vertexes of $x$ are free 
(resp. projective, finitely generated) 
$R$-Modules.    
\end{df}

\begin{nt}[\bf Localization of cubes in the category of Modules] 
\label{df:localization of cubes}
For any $S$-cube $x$ in the category of 
$R$-Modules and a multiplicative subset 
$\cSS\subset R$, $\cSS^{-1}x$ will denote the 
$S$-cube in the category of $\cSS^{-1}R$-Modules 
defined by 
$(\cSS^{-1}x)_T:=\cSS^{-1}(x_T)$ 
for any $T\in \cP(S)$. 
If $\cSS=\{f^n\}_{n\geq 0}$ (reps. $A\ssm \pp$) for some element $f\in R$ 
(resp. a prime ideal $\pp$ in $R$), 
we write $x_f$ (resp. $x_{\pp}$) for $\cSS^{-1}x$. 
\end{nt}

\begin{df} 
\label{nt:cM_A^f(p)}
Let the letter $p$ be a natural number or $\infty$ and 
$I$ an ideal of $A$. 
Let $\cM_A^{I}(p)$ denote the category 
of finitely generated $A$-modules $M$ such 
that $\pd_A M \leqq p$ 
and $\Supp M \subset V(I)$. 
We write $\cM_A^I$ for $\cM_{A}^{I}(\infty)$. 
Since this category is closed under extensions in $\cM_A$, 
it can be considered to be an exact category in the natural way. 
Notice that if $I$ is the zero ideal of $A$, then 
$\cM_A^I(0)$ is just the category $\cP_A$. 
\end{df}

\begin{rem}
\label{rem:HM10} 
In \cite{HM10}, a pseudo-coherent $\cO_X$-Module on a scheme $X$ 
is said to be of {\bf Thomason-Trobaugh weight $q$} if it is 
supported on a regular closed immersion $Y\rinc X$ of codimension $q$ 
and if it is of Tor-dimension $\leq q$. 
The following are equivalent 
for any finitely generated $A$-module $M$ 
and any ideal $I$ which is generated by an $A$-regular sequence 
$f_1,\cdots,f_r$.\\ 
$\mathrm{(1)}$ 
The quasi-coherent $\cO_{\Spec A}$-module 
associated with $M$ is 
of Thomason-Trobaugh weight $r$ supported on $V(I)$.\\ 
$\mathrm{(2)}$ 
$M$ is in $\cM_A^{I}(r)$.\\ 
This is a consequence of the two facts that 
$\Tordim_A M=\pd_A M$ (see \cite[Proposition~4.1.5]{Wei94}) and 
that 
$\Spec A/I \to \Spec A$
is a regular closed immersion of codimension $r$. 
\end{rem}

\sn
In this section, from now on, 
assume that $S$ is a finite set and 
let $\{f_s\}_{s\in S}$ be an $A$-sequence and 
let us fix $0\leqq p\leqq \infty$. 

\begin{df}[\bf Koszul cube]
\label{df:Koszul cube df}
A {\bf Koszul cube} $x$ associated with $\{f_s\}_{s\in S}$ 
is an $S$-cube in 
the category of finitely generated projective $A$-modules $\cP_A$ 
such that 
for each subset $T$ of $S$ and $k$ in $T$, 
$d^k_T$ is an injection and $\coker d^k_T$ is 
in $\cM_A^{f_k A}(1)$. 
We denote the full subcategory of $\Cub^S\cP_A$ 
consisting of 
those Koszul cubes 
associated with $\{f_s\}_{s\in S}$ by $\Kos_A^{\ff_S}$. 
Notice that if $S=\emptyset$, then 
$\Kos_A^{\ff_S}$ is just the category $\cP_A$. 
\end{df}

\begin{ex}
\label{ex:typ Kos cube}
The typical cube associated with $\{f_s\}_{s\in S}$ in 
\ref{Koszul cube def} is obviously 
a non-degenerate free Koszul cube associated with $\{f_s\}_{s\in S}$. 
\end{ex}

\sn
By \ref{ex:A-sequence} $\mathrm{(1)}$, 
we can get the following lemma easily. 

\begin{lem}[\bf Localization of Koszul cubes]
\label{lem:loc of Kos cube}
Let $x$ be a Koszul cube associated with an $A$-sequence 
$\{f_s\}_{s\in S}$ 
and $\pp$ a prime ideal in $A$. 
We put $T:=\{t\in S;f_t\in \pp\}$. 
Then\\ 
$\mathrm{(1)}$ 
$x_{\pp}$ is degenerate along the $t$-direction 
for any $t\in S\ssm T$.\\ 
$\mathrm{(2)}$ 
${x_{\pp}|}_T^{\emptyset}$ is a Koszul cube associated with 
the $A_{\pp}$-sequence $\{f_t\}_{t\in T}$.\\ 
\end{lem}

\begin{lemdf}[\bf Determinant of free Koszul cubes]
\label{lemdf:rank and determinant}
Let $x$ be a free Koszul $S$-cube associated with $\{f_s\}_{s\in S}$. 
Then\\ 
$\mathrm{(1)}$ 
All vertexes of $x$ have same rank.\\ 
$\mathrm{(2)}$ 
By virtue of $\mathrm{(1)}$, 
if we fix bases $\alpha$ of all vertexes of $x$, 
a boundary map 
$d_T^{k,x}$ of $x$ is represented by a square matrix $D_T^{k,x}$ 
and we write $\det_{\alpha} d_T^{k,x}$ or simply $\det d_T^{k,x}$ 
for $\det D_T^{k,x}$. 
For each $T\in \cP(S)$, $k\in T$, 
there is a unit element $u_{T,k}$ such that 
$\det d_S^{x,k}=u_{T,k}\times\det d_T^{x,k}$. 
We call the family $\{\det d_S^{x,k}\}_{k\in S}$ consisting of elements in $A$ 
{\bf determinant} of $x$ 
(with respect to the bases $\alpha$) 
and denote it by $\det_{\alpha} x$ or simply $\det x$. 
If we change bases of $x$ into others, 
then $\det x$ is changing up to units. 
\end{lemdf}

\begin{proof}[\bf Proof] 
First we prove assertion $\mathrm{(1)}$. 
For any $T\in \cP(S)$, $k\in T$, 
since $\coker d_T^{x,k}$ is in $\cM_A^{f_kA}(1)$, 
we learn that $d_T^{x,k}$ induces the isomorphism 
${(x_T)}_{f_k}\isoto{(x_{T\ssm\{k\}})}_{f_k}$. 
This isomorphism implies the equality 
$\rank x_T =\rank x_{T\ssm \{k\}}$. 
Hence we get the result. 
Next we prove assertion $\mathrm{(2)}$. 
We need only check that for any $T\in \cP(S)$ and a pair of 
distinct elements $k$, $k'\in T$, 
there is a unit element $v_{T,k,k'}$ in $A$ such that 
$\det d_T^{x,k}=v_{T,k,k'}\times\det d_{T\ssm\{k'\}}^{x,k}$. 
Since in the total quotient ring of $A$, we have 
the equality 
$\displaystyle{v_{T,k,k'}:=\frac{\det d_T^{x,k}}{\det d_{T\ssm\{k'\}}^{x,k}}
=\frac{\det d^{x,k'}_T}{\det d^{x,k'}_{T\ssm\{k\}}}}$, 
the element $v_{T,k,k'}$ is in $A_{f_k}^{\times}\cap A_{f_{k'}}^{\times}$.
Hence by \ref{lem:localize by reg seq}, 
we learn that $v_{T,k,k'}$ is the desired element.  
\end{proof}

\begin{cordf}[\bf Non-degenerate part of Koszul cubes]
\label{cordf:nondegpart}
Let $x$ be a Koszul $S$-cube. 
If for some $T\in\cP(S)$, $t\in T$, 
$d_T^{x,t}$ is an isomorphism, 
then $x$ is degenerate along the $t$-direction. 
We put 
$N_x:=\{s\in S;\text{$x$ is not degenerate along the $s$-direction.}\}$ 
and call $x_{\nondeg}:=x|_{N_x}^{\emptyset}$ 
the {\bf non-degenerate part} of $x$.
\end{cordf}

\begin{proof}[\bf Proof]
We need to check that 
for any $U\in\cP(S)$ such that $t\in U$, 
$d^{x,t}_U$ is an isomorphism. 
Let us notice that the property of isomorphisms between modules 
is a local property. 
We fix a prime ideal $\pp$ in $A$ and 
put $V:=\{s\in S;f_s\in\pp \}$.
By \ref{lem:loc of Kos cube} and replacing 
$A$, $x$ with $A_{\pp}$, ${(x_{\pp})|}_V$ respectively, 
we shall assume that $A$ is local and $x$ is free Koszul cube. 
Then by \ref{lemdf:rank and determinant} $\mathrm{(2)}$, 
$\det d_T^{x,t}$ is equal to $\det d_U^{x,t}$ 
up to a unit element. 
Hence we get the result.
\end{proof}

\begin{lem}
\label{lem:morphism between free modules}
Let $\phi:R^{\oplus n} \to R^{\oplus n}$ be a homomorphism of $R$-modules 
and assume that 
$\coker \phi$ is annihilated by a non-zero divisor $g \in R$. 
Then\\
$\mathrm{(1)}$ 
There exist an element $b \in R$ and a non-negative integer $m$ 
such that we have the equality $\det \phi \times b=g^m$.\\
$\mathrm{(2)}$ 
$\det \phi$ is a non-zero divisor in $R$.\\
$\mathrm{(3)}$ 
$\phi$ is an injection.
\end{lem}

\begin{proof}[\bf Proof] 
Localizing by $g$, we get the surjection 
$R_g^{\oplus n} \overset{\phi_g}{\rdef} R_g^{\oplus n}$. 
Since in general, a surjective homomorphisms between 
finitely generated free modules with same ranks are isomorphisms, 
it turns out that $\det \phi_g$ is in $R_g^{\times}$. 
Therefore we get assertion $\mathrm{(1)}$ by \ref{lem:localize by reg seq}. 
Since $g$ is a non-zero divisor, $\det \phi$ is also a non-zero divisor. 
Hence we get assertion $\mathrm{(2)}$. 
Let $\phi^{\ast}$ be the adjugate of $\phi$, namely 
$\phi^{\ast}:R^{\oplus n} \to R^{\oplus n}$ is an $R$-module homomorphism 
such that $\phi^{\ast}\phi=\det \phi \id_{R^{\oplus n}}$. 
Since $\det \phi \id_{R^{\oplus n}}$ is an injection, 
we conclude that $\phi$ is also an injection.
\end{proof}

\sn
Recall the definition of non-degenerate cubes from \ref{df:degenerate cube}.

\begin{prop}
\label{prop:det of nondeg free Koszul cubes}
For any non-degenerate free Koszul $S$-cube $x$, 
$\det x$ is an $A$-sequence. 
\end{prop}

\begin{proof}[\bf Proof]
Since $x$ is non-degenerate, for each $s\in S$, $\det d_S^{x,s}$ is not 
a unit element. 
By \ref{lem:morphism between free modules}, 
there are a family of positive integers $\{m_s\}_{s\in S}$ such that 
$\det d_S^{x,s}$ is a divisor of $f_s^{m_s}$ for each $s\in S$. 
Therefore by \ref{ex:A-sequence} $\mathrm{(2)}$ and \ref{lem:A-seq lem}, 
$\det x$ forms an $A$-sequence.
\end{proof}

\begin{thm}
\label{thm:TotstrictKos is 0sph}
For any non-degenerate free Koszul $S$-cube $x$, 
$\Tot x$ is $0$-spherical.
\end{thm}

\sn
To prove the theorem, 
we need to use the Buchsbaum-Eisenbud Theorem \ref{thm:BEthm} below. 

\begin{df}[\bf Fitting ideal]
\label{df:idealminor}
Let $U$ be an $m \times n$ matrix over $A$ 
where $m$, $n$ are positive integers. 
For $t$ in $(\min(m,n)]$ we then denote by $I_t(U)$ the ideal generated by 
the $t$-minors of $U$, that is, 
the determinant of $t\times t$ sub-matrices of $U$.\\ 
For an $A$-module homomorphism 
$\phi:M\to N$ between free $A$-modules of finite rank, 
let us choose a matrix representation $U$ 
with respect to bases of $M$ and $N$. 
One can easily prove that the ideal $I_t(U)$ 
only depends on $\phi$. 
Therefore we put $I_t(\phi):=I_t(U)$ 
and call it the {\bf $t$-th Fitting ideal associated with $\phi$}. 
\end{df}

\begin{nt}[\bf Grade]
\label{nt:grade}
For an ideal $I$ in $A$, we put 
$$S_I:=\{n;\text{There are $f_1,\cdots, f_n\in I$ which forms an $A$-regular sequence.}\}, \text{ and}$$
$$\grade I:=
\begin{cases}
0 & \text{if $S_I=\emptyset$}\\
\max S_I & \text{if $S_I$ is a non-empty finite set}\\ 
+\infty & \text{if $S_I$ is an infinite set}
\end{cases}
.$$
\end{nt}

\begin{thm}[\bf Buchsbaum-Eisenbud \cite{BE73}]
\label{thm:BEthm}
For a complex of free $A$-modules of finite rank.
$$F_{\bullet}:0 \to F_s \onto{\phi_s} F_{s-1} \onto{\phi_{s-1}} \to \cdots \to 
F_1 \onto{\phi_1} F_0 \to 0,$$ 
set $r_i=\overset{s}{\underset{j=i}{\sum}}(-1)^{j-i} \rank F_j$. 
Then the following are equivalent:\\
$\mathrm{(1)}$ $F_{\bullet}$ is $0$-spherical.\\
$\mathrm{(2)}$ $\grade I_{r_i}(\phi_i) \geqq i$ for any $i$ in $(s]$.
\end{thm}

\begin{proof}[\bf Proof of Theorem~\ref{thm:TotstrictKos is 0sph}]
Without loss of generality, 
we may assume that $S=(n]$ and 
$x$ is a free non-degenerate Koszul cube associated with 
an $A$-sequence $f_1,\cdots,f_n$. 
We put $m=\rank x$ and 
$$r_i=\overset{n}{\underset{j=i}{\sum}}(-1)^{j-i} \rank 
\Tot\Typ(\ff_S)_j 
=\overset{n}{\underset{j=i}{\sum}}(-1)^{j-i} 
\begin{pmatrix}
n\\
j
\end{pmatrix}.$$
Then we have 
$$\overset{n}{\underset{j=i}{\sum}}(-1)^{j-i} \rank 
\Tot x=mr_i.$$
It is well-known that 
in this case, 
the complex $\Tot \Typ(\ff_S)$ 
is $0$-spherical. (See \cite[Corollary~4.5.4]{Wei94}). 
Therefore by \ref{thm:BEthm}, 
it follows that $\grade I_{r_i}(d_i^{\Tot\Typ(\ff_S)}) \geq i$ 
for any $i$ in $(n]$. 
Now inspection shows that for each $i\in(n]$, 
$I_{r_i}(d_i^{\Tot\Typ(\ff_S)})\subset I_{mr_i}(d_i^{\Tot x})$. 
Therefore we use Theorem~\ref{thm:BEthm} again, 
it turns out that $x$ is $0$-spherical. 
\end{proof}

\begin{cor}
\label{cor:Kosisadm}
A Koszul cube is admissible.
\end{cor}

\begin{proof}[\bf Proof]
Let $x$ be a Koszul cube associated with 
an $A$-sequence $\{f_s\}_{s\in S}$. 
Since the notion of admissibility 
is a property of certain exactness of 
morphisms of modules, 
we learn that it is a local property. 
We take a prime ideal $\pp$ of $A$ and put 
$T:=\{s\in S;f_s\in \pp\}$. 
Then by \ref{cor:admcriterion II}, 
\ref{lem:loc of Kos cube}, 
\ref{cordf:nondegpart} and replacing $A$, 
$x$ with $A_{\pp}$, ${((x_{\pp})|_{T})}_{\nondeg}$ respectively, 
we shall assume that $A$ is local and 
$x$ is a non-degenerate free Koszul cube. 
We are going to prove the assertion 
by induction of $\# S$ and 
check that $x$ satisfies the conditions 
in \ref{cor:charof adm cube} $\mathrm{(2)}$. 
For $\# S=1$, the assertion is trivial. 
Now we assume that $\# S >1$. 
Since every faces of $x$ are again non-degenerate free Koszul cubes, 
by inductive hypothesis, 
they are admissible. 
On the other hand, 
$\Tot x$ is $0$-spherical by \ref{thm:TotstrictKos is 0sph}. 
Therefore we get the result.  
\end{proof}

\sn
Recall the definition for $\ltimes$ from \ref{df:mult semi-direct prod} 
and $\cM_A^{\ff_T}(\# T)$ from \ref{nt:cM_A^f(p)}. 

\begin{cor}
\label{cor:comp df of Koscube} 
We have the equality 
$$\Kos_A^{\ff_S}=\underset{T\in \cP(S)}{\ltimes} \cM_A^{\ff_T}(\# T).$$
\end{cor}

\begin{proof}[\bf Proof]
For any Koszul cube $x$ associated with 
the $A$-sequence $\{f_s\}_{s\in S}$, 
we need to check the following two assertions.\\ 
$\mathrm{(1)}$ 
$x$ is admissible.\\ 
$\mathrm{(2)}$ 
For any $T\in\cP(S)$ and $U\in\cP(S\ssm T)$, 
$\Homo_0^T(x)_U$ is in $\cM_A^{\ff_T}(\# T)$.\\ 
Assertion $\mathrm{(1)}$ has been proven in \ref{cor:Kosisadm}. 
We prove assertion $\mathrm{(2)}$. 
First let us notice that we have 
$$\Supp\Homo_0^T(x)_U \subset 
\underset{t\in T}{\cap}\Supp\Homo^t_0(x)_U \subset 
\underset{t\in T}{\cap}V(f_t).$$ 
By \ref{cor:homology of admissible cubes}, 
$\Tot (x|_T^U)$ is a projective resolution of 
$\Homo_0^T(x)_U$ and therefore $\pd_A\Homo_0^T(x)_U\leq \# T$. 
This means that $\Homo_0^T(x)_U$ is in 
$\cM_A^{\ff_T}(\# T)$.
\end{proof}

\section{Koszul resolution theorem}
\label{sec:Koszul resolution theorem}

In this section, we will prove 
the first main theorem~\ref{cor:main theorem} 
by utilizing 
theorem~\ref{semi-direct resolution theorem}. 
To check the hypothesis in \ref{semi-direct resolution theorem}, 
the key ingredient 
is Koszul resolution theorem~\ref{Koszul resol thm} which 
gives a resolution process of pure weight modules 
by finite direct sums of typical Koszul cubes. 
Let us recall that $A$ is a commutative noetherian ring with unit 
and in this section, 
let us fix a finite set $S$ 
and an $A$-sequence $\ff_S=\{f_s\}_{s\in S}$ and 
the letter $p$ means a non-negative integer or $\infty$. 

\sn
Recall the definition of $\cM_A^I(q)$ from \ref{nt:cM_A^f(p)}. 

\begin{df}[\bf Reduced modules]
\label{df:reduced modules}
An $A$-module $M$ in $\cM_A^{\ff_S}$ is said to be 
{\bf reduced} if $\ff_SM=0$. 
The full subcategory of reduced $A$-modules is 
just $\cM_{A/\ff_SA}$. 
We write $\cM_{A,\red}^{\ff_S}(p)$ 
for the full subcategory of reduced modules 
in $\cM_A^{\ff_S}(p)$. 
Since $\cM_{A,\red}^{\ff_S}(p)$ 
is closed under taking sub- and quotient objects in $\cM_A^{\ff_S}(p)$, 
applying Lemma~\ref{sub exact cat criterion} below, 
we learn that $\cM_{A,\red}^{\ff_S}(p)$ 
naturally becomes 
an exact category. 
We also write  
$\cM_{A,\red}^{\ff_S}$ 
for $\cM_{A,\red}^{\ff_S}(\infty)$. 
\end{df}

\begin{nt}
\label{nt:indexnonempty}
To emphasize the contrast with the index $\red$, 
we sometimes denote 
$\cM_A^{\ff_S}(p)$, 
$\Kos_A^{\ff_S}$ and so on by 
$\cM_{A,\emptyset}^{\ff_S}(p)$, 
$\Kos_{A,\emptyset}^{\ff_S}$ respectively.
\end{nt}

\begin{lem}
\label{sub exact cat criterion}
Let $\cE$ be an exact category and 
$\cF$ a full subcategory which satisfies 
the following two conditions.\\ 
$\mathrm{(1)}$ 
$\cF$ is closed under taking finite direct sums. 
In particular $\cF$ has a zero object $0$.\\ 
$\mathrm{(2)}$ 
$\cF$ is closed under 
taking admissible sub- and quotient objects in $\cE$. 
That is, for an admissible exact sequence in $\cE$
$$x \rinf y \rdef z,$$
if $y$ is isomorphic to an object in $\cF$, 
then $x$ and $z$ are also.\\ 
Then we can consider $\cF$ as an exact category 
by declaring that a sequence
$x \rinf y \rdef z$
is an admissible exact sequence in $\cF$ if and only if it is in $\cE$.
\end{lem}

\begin{proof}[\bf Proof]
For simplicity, 
we may suppose that 
$\cF$ is closed under isomorphisms in $\cE$. 
Namely for an object $x$ in $\cF$ 
and an object $y$ in $\cE$, 
if $x \isoto y$, then $y$ is also in $\cF$. 
We shall just check that 
the class of admissible monomorphisms 
(resp. admissible epimorphisms) in $\cF$ is 
closed under compositions 
and co-base (resp. base) change by arbitrary morphisms. 
We only check for the admissible monomorphisms case. 
For the admissible epimorphisms case, almost same arguments are working.\\ 
Let $x \rinf y$ and $y \rinf z$ be admissible monomorphisms in $\cF$. 
By $\mathrm{(2)}$, 
$z/x$ can be taken in $\cF$. 
Therefore the sequence 
$$x \rinf z \rdef z/x$$
is an admissible exact sequence in $\cE$ and 
the composition $x \rinf z$ is an admissible monomorphism in $\cF$.\\
Next consider the following commutative diagram in $\cE$:
$$\xymatrix{
x \ar@{>->}[r] \ar[d] \ar@{}[rd]|{\bigstar} & 
y \ar@{->>}[r] \ar[d] & z \ar@{=}[d]\\
x' \ar@{>->}[r] & y' \ar@{->>}[r] & z
}$$
where the square $\bigstar$ is coCartesian and 
we assume that $x$, $y$, $z$ and $x'$ are in $\cF$. 
Then by \cite[p.406 step 1]{Kel90}, 
we have an admissible exact sequence
$$x \rinf x' \oplus y \rdef y'.$$
By $\mathrm{(1)}$, 
$x' \oplus y$ is in $\cF$ 
and by $\mathrm{(2)}$, 
$y'$ is also in $\cF$. 
\end{proof}

\sn
Utilizing \ref{cor:mult semi is exact}, 
we can easily get the following lemma. 
Recall the definition of $\rtimes$ from \ref{df:mult semi-direct prod}. 

\begin{df}
\label{lemdf:cMAfTfsp}
Let $S=U\coprod V$ be a disjoint decomposition of $S$.\\
$\mathrm{(1)}$ 
We define the categories $\cM_A(\ff_U;\ff_V)(p)$, 
$\cM_{A,\red}(\ff_U;\ff_V)(p)$ which 
are full subcategories of $\Cub^V\cM_A$ by 
$$\cM_{A,?}(\ff_U;\ff_V)(p)=
\underset{T\in\cP(V)}{\ltimes}\cM_{A,?}^{\ff_{U\coprod T}}(p+\# T)$$
where $?=\emptyset$ or $\red$.
In particular, 
we write $\Kos_{A,\red}^{\ff_S}$ 
for  $\cM_{A,\red}(\ff_{\emptyset};\ff_S)(0)$. 
This notation is compatible with the equality 
$\mathrm{(\ref{equ:Kos as MAfUfV})}$ 
in \ref{rem:cMAfTfsp} $\mathrm{(3)}$. 
A cube in $\Kos_{A,\red}^{\ff_S}$ 
is said to be a {\bf reduced Koszul cube} 
(associated with an $A$-sequence $\{f_s\}_{s\in S}$).\\
$\mathrm{(2)}$ 
A morphism $f:x\to y$ in $\cM_{A,?}(\ff_U;\ff_V)(p)$ is 
a ({\bf total}) {\bf quasi-isomorphism} if $\Tot f$ is 
a quasi-isomorphism. 
We write 
$\tq(\cM_{A,?}(\ff_U;\ff_V)(p))$ or simply $\tq$ 
for the class of total quasi-isomorphisms in $\cM_{A,?}(\ff_U;\ff_V)(p)$. 
\end{df}

\begin{rem}
\label{rem:cMAfTfsp}
In the notation above, we have the following.\\
$\mathrm{(1)}$ 
$\cM_A(\ff_U;\ff_V)(p)$ is closed under extensions in 
$\Cub^V\cM_A$. 
In particular it becomes 
an exact category in the natural way.\\ 
$\mathrm{(2)}$ 
$\cM_{A,\red}(\ff_U;\ff_V)(p)\rinc \cM_A(\ff_U;\ff_V)(p)$ 
is closed under taking finite direct sums and 
admissible sub- and quotient objects. 
In particular, 
$\cM_{A,\red}(\ff_U;\ff_V)(p)$ 
naturally becomes an exact category 
by virtue of \ref{sub exact cat criterion}.\\
$\mathrm{(3)}$ 
By \ref{cor:comp df of Koscube}, we have the equality 
\begin{equation}
\label{equ:Kos as MAfUfV}
\cM_A(\ff_{\emptyset};\ff_S)(0)=\Kos_A^{\ff_S}. 
\end{equation}
$\mathrm{(4)}$ 
By definitions of multi semi-direct products 
in \ref{df:mult semi-direct prod} $\mathrm{(1)}$ 
and 
$\cM_{A,?}(\ff_U;\ff_V)(p)$, 
cubes in $\cM_{A,?}(\ff_U;\ff_V)(p)$ are admissible. 
A morphism $f:x\to y$ in $\cM_{A,?}(\ff_U;\ff_V)(p)$ is 
a total quasi-isomorphism if 
and only if 
$\Homo_0^V(f)$ is an isomorphism by 
\ref{prop:homology of admissible cubes} $\mathrm{(2)}$.\\
$\mathrm{(5)}$ 
We put $\cF_T:=\cM_A^{\ff_T}(p+\# T)$ 
for any $T\in \cP(S)$ 
and $\fF=\{\cF_T\}_{T\in \cP(S)}$ and 
for any disjoint 
decomposition $S=U\coprod V$, 
we put $\fF|_V^U:=\{\cF_{U\coprod T} \}_{T\in \cP(V)}$. 
Then since 
we have the equality 
$\ltimes \fF|_V^U=\cM_{A,?}(\ff_U;\ff_V)(p+\# U)$, 
the family $\fF$ satisfies the condition \ref{rem:exactness}. 
\end{rem}

\begin{rem}
\label{rem:inductive describe cMA(fT;fS)}
Let $S=U\coprod V$ be a disjoint decomposition of $S$ with $V\neq \emptyset$. 
Then for any $v\in V$, 
by \ref{prop:rel multi and mono semi-direct prod} 
$\mathrm{(1)}$ and $\mathrm{(3)}$, 
we have the following equalities
$$\cM_{A,?}(\ff_U;\ff_V)(p)=
\cM_{A,?}(\ff_{U\coprod\{v\}};\ff_{V\ssm\{v\}})(p+1)
\ltimes \cM_{A,?}(\ff_U;\ff_{V\ssm\{v\}})(p)$$
$$
{\cM_{A,?}(\ff_U;\ff_V)(p)}^{\tq}=
{\cM_{A,?}(\ff_{U\coprod\{v\}};\ff_{V\ssm\{v\}})(p+1)}^{\tq}
\ltimes \cM_{A,?}(\ff_U;\ff_{V\ssm\{v\}})(p)
$$
for $?=\emptyset$ or $\red$.
\end{rem}

\sn
To prove Proposition~\ref{weight resolution theorem}, 
we need to recall the following facts.

\begin{rev}
\label{pd criterion} 
Notice that for a short exact sequence of $A$-modules
$$0 \to N \to N' \to N'' \to 0,$$
we can easily prove the following assertions 
by utilizing the $\Ext$-criterion of projective dimensions. 
(see \cite[pd Lemma~4.1.6]{Wei94}).\\
$\mathrm{(1)}$ 
If $\pd_A N'' \leqq n+1$ and $\pd_A N' \leqq n$, then $\pd_A N \leqq n$.\\ 
$\mathrm{(2)}$ 
If $\pd_A N \leqq n$ and $\pd_A N'' \leqq n$, then $\pd_A N' \leqq n$. 
\end{rev}

\begin{prop}
\label{weight resolution theorem}
Let $n$ be an integer such that 
$\# S \leqq n < \infty$ 
and $?=\emptyset$ or $\red$. 
Then\\ 
$\mathrm{(1)}$ 
$\cM_{A,?}^{\ff_S}(n)$ is closed under 
extensions and taking kernels 
of admissible epimorphisms 
in $\cM_{A,?}^{\ff_S}(n+1)$.\\ 
$\mathrm{(2)}$ 
Moreover the inclusion functor of opposite categories 
${(\cM_{A,?}^{\ff_S}(n))}^{\op} \rinc {(\cM_{A,?}^{\ff_S}(n+1))}^{\op}$ 
satisfies 
the resolution conditions in \ref{resol cond df}. 
In particular, 
we have a homotopy equivalence on $K$-theory:
$$K(\cM_{A,?}^{\ff_S}(n)) \to 
K(\cM_{A,?}^{\ff_S}(n+1)).$$ 
\end{prop}

\begin{proof}[\bf Proof]
For $\mathrm{(1)}$, 
we need only check the projective dimension conditions 
and they are easily done by 
$\mathrm{(1)}$ and $\mathrm{(2)}$ in \ref{pd criterion}. 
For $\mathrm{(2)}$, 
we need to prove the following two assertions 
for $?=\emptyset$ and $\red$ respectively. 

\sn
$\mathrm{(a)}$ For any $M \in \cM_A^{\ff_S}(n+1)$, 
there exists an $A$-module $N$ in $\cM_{A,?}^{\ff_S}(n)$ and 
an admissible epimorphism $N \rdef M$. 

\sn
$\mathrm{(b)}$ For any admissible short exact sequence 
$L \rinf N \rdef M$ 
in $\cM_{A,?}^{\ff_S}(n+1)$, 
if $N$ and $M$ are in $\cM_{A,?}^{\ff_S}(n)$, 
then $L$ is also in $\cM_{A,?}^{\ff_S}(n)$. 

\sn
Proof of assertion $\mathrm{(a)}$: 
For any $M \in \cM_A^{\ff_S}(n)$, 
there are positive integers $t_s$ such that 
$f_s^{t_s}M=0$ for any $s\in S$. 
We put $g_s:=f_s^{t_s}$ and $B:=A/J$ 
where $J$ is the ideal generated by $\{g_s\}_{s\in S}$. 
Therefore we can consider $M$ to 
be a finitely generated $B$-module 
and there is a surjection 
from a finitely generated free $B$-module $N$ to $M$. 
Since $\pd_A N=\# S$, 
$N$ is in $\cM_A^{\ff_S}(n)$. 
If $M$ is reduced, 
we can take $t_s=1$ for each $s\in S$. 
Then in this case $N$ is also reduced.

\sn
Proof of assertion $\mathrm{(b)}$: 
We get the result by \ref{pd criterion} $\mathrm{(1)}$
for $?=\emptyset$. 
If $N$ is reduced, then $L$ is also reduced. 
\end{proof}

\begin{cor}
\label{practical closed conditions}
For $?=\emptyset$ or $\red$, any disjoint decomposition 
$S=U\coprod V$ and any integer $p\ge \# U$, 
$\cM_{A,?}(\ff_U;\ff_V)(p)$ is closed under 
extensions and taking kernels 
of admissible epimorphisms in 
$\cM_{A,?}(\ff_U;\ff_V)(p+1)$. 
\end{cor}

\begin{proof}[\bf Proof]
By virtue of \ref{cor:mult semi is exact} $\mathrm{(2)}$, 
we need only check that 
$\cM_{A,?}^{\ff_{T\coprod U}}(p+\# T)$ is 
closed under extensions 
and taking kernels of admissible epimorphisms 
in $\cM_{A,?}^{\ff_{T\coprod U}}(p+1+\# T)$ for any $T\in\cP(V)$. 
This was done in \ref{weight resolution theorem}. 
\end{proof}

\begin{rem}
\label{rem:regularclosedimmerion}
Let us assume that $A$ is Cohen-Macaulay. 
Recall that a commutative noetherian ring $C$ is 
Cohen-Macaulay if and only if every ideal of 
height at least $p$ contains an $C$-regular sequence of length $p$. 
(See \cite[\S 2.5, Proposition.\ 7]{Bou98}.) 
Notice that the ordered set $X$ of all ideals of $A$ that contains an $A$-regular sequence 
of length $p$ with usual inclusion is directed. 
Therefore 
$\cM^p_A(p)$ is the filtered limit 
$\varinjlim_{\fg_S}\cM^{\fg_S}_A(p)$ 
where $\fg_S$ runs through any regular sequence 
such that $\# S=p$. 
\end{rem}

\begin{cor}
\label{thm:HM10}
Let us assume that $A$ is regular. 
Then for any natural number $p$, 
the inclusion functor $\cM^p_A(p) \rinc \cM_A^p$ induces 
a homotopy equivalence on $K$-theory: 
$$K(\cM^p_A(p)) \to K(\cM_A^p).$$
\end{cor}

\begin{proof}[\bf Proof]
By regularity of $A$, 
we may ignore the projective dimension 
condition in \ref{weight resolution theorem} $\mathrm{(2)}$. 
The assertion follows from 
\ref{weight resolution theorem} $\mathrm{(2)}$ and 
\ref{rem:regularclosedimmerion}. 
\end{proof}

\begin{thm}[\bf Koszul resolution theorem]
\label{Koszul resol thm}
Let $n$ be a non-negative integer, 
$S=U\coprod V$ a disjoint decomposition of $S$ and  $p\geq \# U$ 
an integer. 
Fix an object
$z\in\HOM([n],\cM_A(\ff_U;\ff_V)(p+1))$, 
$$z:z(0)\to z(1)\to \cdots \to z(n).$$
$\mathrm{(1)}$ 
For each $s\in S$, 
there is a family of 
non-negative integers $\{m_s\}_{s\in S}$ 
such that $f_u^{m_u}{z(j)}_T=0$ 
for any $T\in \cP(V)$, 
$j\in [n]$ 
and $u\in U$ and 
$f_v^{m_v}\Homo_0(\Tot(z(j)))=0$ for 
any $j\in [n]$ and $v\in V$. 
We put $g_s=f_s^{m_s}$ for each $s\in S$ 
and $B:=A/\fg_U$ 
where $\fg_U$ is the ideal generated by $\{g_u\}_{u\in U}$.\\ 
$\mathrm{(2)}$ 
There is an object 
$y \in \HOM([n],\cM_A(\ff_U;\ff_V)(\# U))$ 
such that for each $i\in [n]$, $y(i)$ 
is of the following form: 
$$ 
y(i)=\underset{T\in\cP(V)}{\bigoplus} 
{\Typ_{B}({\fg}^T_V)}^{\oplus l_T(i)} 
$$
and such that there is an admissible epimorphism $y \rdef z$. 
Here the notation ${\fg}^T_V$ means the family
$\{g^{\chi_T(v)}_v\}_{v\in V}$ where 
$\chi_T$ is the characteristic function associated with $T$ 
{\rm (}see \ref{rem:cubentrem}{\rm )}
and for the definition of 
the typical Koszul cube $\Typ$, see \ref{Koszul cube def}. 
\end{thm}

\begin{proof}[\bf Proof]
Since ${z(j)}_T$ is in $\cM_A^{\ff_U}$ 
for any $T\in\cP(V)$ 
and 
$\Homo_0(\Tot(z(j)))$ is in $\cM_A^{\ff_{S}}$, 
assertion $\mathrm{(1)}$ is trivial. 
We are concentrating on proving assertion $\mathrm{(2)}$.  
We first prove that for the case $n=0$ 
by induction on $\# V$. 
In this case, we consider $z$ 
to be an object in $\cM_A(\ff_U;\ff_V)(p+1)$. 
Let us assume that $V=\emptyset$. 
Since $z$ is a finitely generated $B$-module, 
there is an integer $l_1$ 
and a surjection $B^{\oplus l_1} \to z$. 
We put $y=B^{\oplus l_1}$ and by \ref{weight resolution theorem}, 
the map $y \to z$ is an admissible epimorphism. 
Next we assume that $\# V\geq 1$ and 
let us fix an element $v\in V$ 
and an object $z \in \cM_A(\ff_U;\ff_V)(p+1)$. 
Then by the formula \ref{rem:inductive describe cMA(fT;fS)}, 
we can consider $z$ to be a complex $[z_1 \onto{d^z} z_0]$ 
in $\cM_A(\ff_U;\ff_{V\ssm\{v\}})(p+1)$. 
By the inductive hypothesis, 
there is an admissible epimorphism $y' \rdef z_0$ where $y'$ is of the form 
$$y'=\underset{T\in \cP(V\ssm\{v\})}{\bigoplus} 
{\Typ_B(\fg^{T}_{V\ssm\{v\}})}^{\oplus l_{1,T}} . 
$$
Therefore by the cube lemma~\ref{cube lemma}, 
we get a term-wised surjection morphism $y'' \to \Homo_0z$ 
where $y''$ is of the form 
$$y''=\underset{T\in \cP(V\ssm\{v\})}{\bigoplus} 
{\Typ_{B/(g_{v})}(\fg^{T}_{V\ssm\{v\}})}^{\oplus l_{1,T}} 
$$ 
and it makes the diagram below commutative:
$$\footnotesize{\xymatrix{
y' \ar@{->>}[r] \ar@{->>}[d] & z_0 \ar@{->>}[d]\\ 
y'' \ar@{->>}[r] & \Homo_0 z
}}$$
where the vertical maps are the canonical projections. 
Therefore we get a map $y' \to z_1$ 
which makes the diagram below commutative:
$$\footnotesize{\xymatrix{
y' \ar[r] \ar[d]_{g_{v}} & z_1 \ar[d]^{d^z}\\ 
y' \ar@{->>}[r] & z_0  .
}}$$ 
By the induction hypothesis, 
there is a term-wised surjective morphism $y'\oplus y''' \to z_1$ 
where $y'''$ is of the form
$$y'''=\underset{W\in\cP(V\ssm\{v\})}{\bigoplus} 
{\Typ_B(\fg^{W}_{V\ssm\{v\}})}^{\oplus l_{0,W}} $$
and it makes the diagram below commutative:
$$\footnotesize{\xymatrix{
y_1 \ar@{->>}[r] \ar[d]_{d^y} & z_1 \ar[d]^{d^z}\\ 
y_0 \ar@{->>}[r] & z_0  
}}$$ 
where $y_i=y'\oplus y'''$ for $i=0$, $1$ and 
$\displaystyle{d^y=\footnotesize{\begin{pmatrix} 
g_v & 0 \\ 
0 & 1 
\end{pmatrix}}}$ and we put $y=[y_1 \onto{d^y} y_0]$. 
Thus by \ref{practical closed conditions}, 
we learn that $y \to z$ is an admissible epimorphism 
and therefore we get the conclusion for the case of $n=0$. 
Next we consider the case of general $n$. 
For each $z$ and each $i\in [n]$, 
by the previous argument, 
we have 
$y(i)=\underset{T\in \cP(V)}{\bigoplus} 
{\Typ_B(\fg^{T}_{V})}^{\oplus l_{T}(i)}$ 
for a suitable non-negative integer $l_{T}(i)$ 
and an admissible epimorphism $y(i) \rdef z(i)$. 
So we need only prove that for each 
$i\in [n-1]$, 
there is a map $y(i) \to y(i+1)$ 
which makes diagram below commutative:
$$\footnotesize{\xymatrix{
y(i) \ar@{-->}[r] \ar@{->>}[d] 
& y(i+1) \ar@{->>}[d] \\ 
z(i) \ar[r] 
& z(i+1) .
}}$$
Since $y(i)$ is in $\underset{T\in \cP(V)}{\ltimes}\cP_{B/\fg_T}$, 
by applying \ref{cor:mult semi is exact} $\mathrm{(3)}$ to 
$\cF_T=\cP_{B/\fg_T}$ and $\cE_T=\cM_{B,\red}^{\fg_T}$ 
for each $T\in \cP(V)$, 
we learn that $y(i)$ is a projective object in 
$\underset{T\in \cP(V)}{\ltimes}\cM_{B,\red}^{\fg_T}$. 
Moreover since $y(i+1)$, $z(i+1)$ are in 
$\underset{T\in \cP(V)}{\ltimes}\cM_{B,\red}^{\fg_T}$, 
there is the dotted map in the commutative diagram above 
by projectivity of $y(i)$. 
\end{proof}

\sn
Recall the definition of strongly adroit systems from \ref{df:adroit system} 
and the definition of $\tq$ from \ref{lemdf:cMAfTfsp}.

\begin{cor}
\label{cor:adroit system}
For $?=\emptyset$ or $\red$, any decomposition $S=U\coprod V$ 
with $V\neq \emptyset$, any element $v$ in $V$ 
and 
any integer $p\geq \# U$, 
triples 
$$\cX=(\cM_{A,?}(\ff_{U};\ff_{V\ssm\{v\}})(p),
\cM_{A,?}(\ff_U;\ff_{V\ssm\{v\}})(p+1),
\cM_{A,?}(\ff_{U\coprod\{v\}};\ff_{V\ssm\{v\}})(p+1)) \text{  and}$$ 
$$\cX'=(\cM_{A,?}(\ff_{U};\ff_{V\ssm\{v\}})(p),
\cM_{A,?}(\ff_U;\ff_{V\ssm\{v\}})(p+1),
\cM_{A,?}(\ff_{U\coprod\{v\}};\ff_{V\ssm\{v\}})(p+1)^{\tq})$$ 
are strongly adroit systems. 
\end{cor}

\begin{proof}[\bf Proof] 
Consider the triple 
\begin{align*}
\cE_1:&=\cM_{A,?}(\ff_U;\ff_{V\ssm\{v\}})(p),\\ 
\cE_2:&=\cM_{A,?}(\ff_U;\ff_{V\ssm\{v\}})(p+1)\ \ \text{and}\\ 
\cF:&=\cM_{A,?}(\ff_{U\coprod\{v\}};\ff_{V\ssm\{v\}})(p+1). 
\end{align*}

\begin{claim}
$\cF$ is contained in $\cE_2$. 
\end{claim}

\begin{proof}[\bf Proof of claim]
If $V=\{v\}$, 
then $\cE_2=\cM_{A,?}^{\ff_U}(p+1)$, 
$\cF=\cM_{A,?}^{\ff_{U\coprod\{v\}}}(p+1)$ and 
therefore we get the assertion. 
If $\#V\geq 2$, then let us fix an element $v'\in V\ssm\{v\}$. 
Then by \ref{rem:inductive describe cMA(fT;fS)}, 
we have equalities 
$$\cE_2=\cM_{A,?}((\ff_{U\coprod\{v'\}};\ff_{V\ssm\{v,\ v'\}})(p+2)
\ltimes 
\cM_{A,?}(\ff_{U\coprod\{v'\}};\ff_{V\ssm\{v,\ v'\}})(p+1)\ \ \text{and}$$ 
$$\cF=\cM_{A,?}(\ff_{U\coprod\{v,\ v'\}};\ff_{V\ssm\{v,\ v'\}})(p+2)
\ltimes
\cM_{A,?}(\ff_{U\coprod\{v,\ v'\}};\ff_{V\ssm\{v,\ v'\}})(p+2).$$ 
Hence we learn that 
$\cF$ is contained in $\cE_2$. 
\end{proof}

\sn
Since $\cF$ is closed under extensions 
(resp. sub- and quotient objects) 
in $\Cub^{V\ssm\{v\}}\cM_A$ if $?=\emptyset$ (resp. $?=\red$), 
$\cF\ltimes\cE_i$ ($i=1$, $2$) are strict exact subcategories 
of $\Cub^V\cM_A$ by \ref{semidiret is exact}. 
The conditions {\bf (Adr 2)} and {\bf (Adr 3)} 
follow from \ref{practical closed conditions}. 
Finally the condition {\bf (Adr 5)} 
follows from \ref{Koszul resol thm}. 
For the proof for $\cX'$ is similar. 
Therefore we get the result. 
\end{proof}

\begin{cor}
\label{cor:main theorem}
For $?=\emptyset$ or $\red$, any decomposition $S=U\coprod V$ 
and any integer $p\geq \# U$, 
we have the following.\\ 
$\mathrm{(1)}$ 
{\bf (Local weight theorem).} 
Assume that $V$ is a non-empty set and $v$ is an element in $V$. 
Then the exact functor 
$\Homo_0^v:(\cM_{A,?}(\ff_U;\ff_V)(p),\tq) \to 
(\cM_{A,?}(\ff_{U\coprod\{v\}};\ff_{V\ssm\{v\}})(p+1),\tq)$ 
induces a homotopy equivalence on $K$-theory:
$$K(\Homo_0^v):K(\cM_{A,?}(\ff_U;\ff_V)(p);\tq) \to 
K(\cM_{A,?}(\ff_{U\coprod\{v\}};\ff_{V\ssm\{v\}})(p+1);\tq).$$
In particular the exact functor 
$\Homo_0\Tot:(\cM_{A,?}(\ff_U;\ff_V)(p),\tq) \to \cM_{A,?}^{\ff_S}(p+\# V)$ 
induces a homotopy equivalence on $K$-theory:
$$K(\cM_{A,?}(\ff_U;\ff_V)(p),\tq) \to K(\cM_{A,?}^{\ff_S}(p+\# V)).$$
In particular the exact functor 
$\Homo_0\Tot:(\Kos_{A,?}^{\ff_S},\tq)\to \cM_{A,?}^{\ff_S}(\# S)$ 
induces a homotopy equivalence on $K$-theory: 
$$K(\Kos_{A,?}^{\ff_S},\tq) \to K(\cM_{A,?}^{\ff_S}(\# S)).$$
$\mathrm{(2)}$ 
The exact functors 
$$\lambda:\cM_{A,?}(\ff_U;\ff_V)(p)\to 
\prod_{T\in\cP(V)}\cM_{A,?}^{\ff_{U\coprod T}}(p+\# T)\text{  and}$$ 
$$\lambda':{\cM_{A,?}(\ff_U;\ff_V)(p)}^{tq}\to 
\prod_{T\in\cP(V)\ssm\{V\}}\cM_{A,?}^{\ff_{U\coprod T}}(p+\# T) $$ 
which sends an object $x$ 
to $(\Homo_0^{U\coprod T}(x))_{T\in\cP(V)} $ and 
$(\Homo_0^{U\coprod T}(x))_{T\in\cP(V)\ssm\{V\}}$ respectively  
induce homotopy equivalences on $K$-theory:
$$K(\cM_{A,?}(\ff_U;\ff_V)(p))\to 
\prod_{T\in\cP(V)}K(\cM_{A,?}^{\ff_{U\coprod T}}(p+\# T)) {\text{   and}}$$
$$
K({\cM_{A,?}(\ff_U;\ff_V)(p)}^{\tq})\to 
\prod_{T\in\cP(V)\ssm\{V\}}K(\cM_{A,?}^{\ff_{U\coprod T}}(p+\# T)).$$
$\mathrm{(3)}$ 
{\bf (Split fibration theorem).} 
The inclusion functors 
${\cM_{A,?}(\ff_U;\ff_V)(p)}^{\tq} \rinc \cM_{A,?}(\ff_U;\ff_V)(p)$ 
and the identity functor on $\cM_{A,?}(\ff_U;\ff_V)(p)$ 
induce a split fibration sequence: 
$$K({\cM_{A,?}(\ff_U;\ff_V)(p)}^{\tq}) \to K(\cM_{A,?}(\ff_U;\ff_V)(p)) \to K(\cM_{A,?}(\ff_U;\ff_V)(p);\tq).$$
In particular we have a split fibration sequence: 
$$K({\Kos_{A,?}^{\ff_S}}^{\tq}) \to K(\Kos_{A,?}^{\ff_S}) \to K(\Kos_{A,?}^{\ff_S};\tq) .$$
\end{cor}

\begin{proof}[\bf Proof]
Proof of assertion $\mathrm{(1)}$: 
Consider the strongly adroit system $\cX=(\cE_1,\cE_2,\cF)$ in (the proof of) \ref{cor:adroit system}. 
We have an equality $\cF\ltimes \cE_1=\cM_{A,?}(\ff_U;\ff_V)(p)$ 
by \ref{rem:inductive describe cMA(fT;fS)}. 
We apply \ref{semi-direct resolution theorem} $\mathrm{(2)}$ to $\cX$ and 
obtain the proof. 
The second assertion comes from the isomorphism $\Homo_0\Tot\isoto\Homo_0^V$ 
by \ref{rem:exactness}, 
\ref{rem:cMAfTfsp} $\mathrm{(5)}$ and the first assertion. 
The third result follows from the equality 
$\mathrm{(\ref{equ:Kos as MAfUfV})}$ in \ref{rem:cMAfTfsp} $\mathrm{(3)}$ 
and the second assertion. 

\sn
Proof of assertion $\mathrm{(2)}$: 
We proceed by induction on the cardinality of $V$. 
If $V=\emptyset$, 
then $\cM_{A,?}(\ff_U;\ff_V)(p)=\cM_{A,?}^{\ff_U}(p)$ and 
${\cM_{A,?}(\ff_U;\ff_V)(p)}^{\tq}=\{0\}$. 
Therefore the assertion is trivial. 
If $\# V \geq 1$ and $v$ is an element in $V$, 
then by the equalities in \ref{rem:inductive describe cMA(fT;fS)} 
$\lambda$ and $\lambda'$ have the factorizations
$${\scriptstyle{K(\cM_{A,?}(\ff_U;\ff_V)(p)) 
\onto{\textbf{I}} 
K(\cM_{A,?}(\ff_{U\coprod\{v\}};\ff_{V\ssm\{v\}})(p+1))\times 
K(X) \onto{\textbf{II}}\!\!\!\!\!\!\! 
\underset{T\in\cP(V)}{\prod}\!\!\!\!\!\!K(\cM_{A,?}^{\ff_{U\coprod T}}(p+\# T))}} $$
$${\scriptstyle{K({\cM_{A,?}(\ff_U;\ff_V)(p)}^{\tq}) 
\onto{\textbf{I}} 
K({\cM_{A,?}(\ff_{U\coprod\{v\}};\ff_{V\ssm\{v\}})(p+1)}^{\tq})\times 
K(X) \onto{\textbf{II}}\!\!\!\!\!\! 
\underset{T\in\cP(V)\ssm\{V\}}{\prod}\!\!\!\!\!\!K(\cM_{A,?}^{\ff_{U\coprod T}}(p+\# T))}}$$
where 
$X$ denotes $\cM_{A,?}(\ff_U;\ff_{V\ssm\{v\}})(p)$, and 
the maps \textbf{I} and \textbf{II} are homotopy equivalences by 
\ref{semi-direct resolution theorem} $\mathrm{(3)}$ and 
the inductive hypothesis respectively. 
Hence we get the result.

\sn
Proof of assertion $\mathrm{(3)}$: 
Let us consider the commutative diagram below:
$${\footnotesize{\displaystyle{\xymatrix{
K({\cM_{A,?}(\ff_U;\ff_V)(p)}^{\tq}) \ar[r] \ar[d]_{K(\lambda')}^{\wr} & 
K(\cM_{A,?}(\ff_U;\ff_V)(p))\ar[r] \ar[d]_{K(\lambda)}^{\wr} & 
K(\cM_{A,?}(\ff_U;\ff_V)(p);\tq)\ar[d]^{K(\Homo_0\Tot)}_{\wr}\\
\underset{T\in\cP(V)\ssm\{V\}}{\prod}\!\!\!\!\!\!\!\!\!\!\!K(\cM_{A,?}^{\ff_{U\coprod T}}(p+\# T)) \ar[r] & 
\underset{T\in \cP(V)}{\prod}\!\!\!\!K(\cM_{A,?}^{\ff_{U\coprod T}}(p+\# T)) 
\ar[r] & 
K(\cM_{A,?}^{\ff_S}(p+\# V)).
}}}}$$
Here vertical maps are homotopy equivalence by $\mathrm{(1)}$ 
and $\mathrm{(2)}$ and the bottom horizontal line is 
a split fibration sequence. 
Hence we get the result.
\end{proof}

\sn
Recall the definition of $\cM_A^I(p)$ from \ref{nt:cM_A^f(p)} 
and $\Kos_A^{\ff_S}$ from \ref{df:Koszul cube df}.

\begin{proof}[\bf Proof of Theorem~\ref{thm:intromaincor}]
Since $A$ is local, 
every $A$-regular sequence is an $A$-sequence 
by \ref{ex:A-sequence} $\mathrm{(4)}$ $\mathrm{(i)}$. 
Let us assume that $\# S=p$ and let 
$\ff_S=\{f_s\}_{s\in S}$ be an $A$-sequence. 
Then the exact functor 
$\Homo_0\Tot:(\Kos_A^{\ff_S},\tq) \to (\cM_A^{\ff_S}(p),i)$ 
induces a homotopy equivalence on $K$-theory 
by \ref{cor:main theorem} $\mathrm{(1)}$. 
On the other hand we have the homotopy equivalence 
$K(\cM_A^{p}(p)) \isoto 
\varinjlim_{\fg_S} K(\cM_A^{\fg_S}(p))$ 
where $V(\fg_S)\rinc \Spec A$ runs over the regular closed immersion of codimension $p$ 
by \ref{rem:regularclosedimmerion}. 
Therefore the Grothendieck group $K_0(\cM_A^p(p))$ 
is generated by modules of the form 
$$\Homo_0(\Tot x)\isoto\coker(\bigoplus_{i=1}^px_{\{i\}} 
\onto{\begin{pmatrix}d^{x,1} &\cdots & d^{x,p}\end{pmatrix}} x_{\emptyset})
\isoto
x_{\emptyset}/<\im d^{x,1}_{\{1\}},\cdots,
\im d^{x,p}_{\{p\}}>$$
where $x$ is a non-degenerate free Koszul cube associated with 
some $A$-sequence $g_1,\cdots,g_p$ and 
$d^{x,i}_{\{i\}}:x_{\{i\}} \to x_{\emptyset}$ is a boundary morphism of $x$. 
Since 
the sequence 
$\det d^{x,1}_{\{1\}},\cdots,\det d^{x,p}_{\{p\}}$ forms 
an $A$-sequence 
by \ref{prop:det of nondeg free Koszul cubes}, 
we obtain the result. 
\end{proof}

\section{A d\'evissage theorem for $K$-theory of Koszul cubes on regular rings}
\label{sec:dev thm}

In this section, 
we assume that $A$ is a commutative regular noetherian ring with unit 
and that 
the global homological dimension of $A$ is $n$ 
and $S$ is a finite set. 
Let us fix an $A$-sequence $\{f_s\}_{s\in S}$ 
and let $I$ be the ideal in $A$ generated by $\{f_s\}_{s\in S}$. 
The aim of this section is 
to prove a d\'evissage theorem~\ref{cor:Devissage 3} 
for Koszul cubes on $A$. 

\sn
Recall from 
\ref{nt:cM_A^f(p)} and \ref{df:reduced modules} that 
$\cM_{A}^I(p)$ is the category of finitely generated $A$-modules 
of projective dimension $\leq p$ and $\Supp M\subset V(I)$ and 
$\cM_{A,\red}^I(p)$ is the full subcategory of modules $M$ with 
$IM=0$, and that $\# S=\pd_A A/I$. 

\begin{prop}
\label{cor:Devissage 2} 
For any integer $p\geqq\# S$, 
the inclusion functor 
$\iota:\cM_{A,\red}^{I}(p) \rinc \cM_{A}^{I}(p)$ 
induces a homotopy equivalence on $K$-theory:
$$K(\iota):K(\cM_{A,\red}^{I}(p)) \to 
K(\cM_A^{I}(p)).$$
\end{prop}

\begin{proof}[\bf Proof]
First assume that $p\geq n$. 
Then every $A$-module $M$ has $\pd M\leq n$, 
so $\cM_{A}^I(p)=\cM_{A}^I$ and $\cM_{A,\red}^I=\cM_{A/I}$. 
In this case, 
the result was proven by Quillen in \cite{Qui73}. 

\sn
Next we assume that $n\geq p\geq\# S$.
Then the inclusion functors 
$\cM_{A,?}^I(p) \rinc \cM_{A,?}^I(n)$ and 
$\cM_{A,\red}^I(k) \rinc \cM_A^I(k)$ $(k=p$, $n)$ 
yield the commutative diagram below: 
$${\footnotesize{\xymatrix{
K(\cM_{A,\red}^I(p)) \ar[r] \ar[d]_{\wr} & K(\cM_A^I(p))\ar[d]^{\wr}\\
K(\cM_{A,\red}^I(n)) \ar[r]^{\ \ \ \sim} & K(\cM_A^I(n)).
}}}$$
Here the vertical maps and the bottom horizontal map 
are homotopy equivalences by 
\ref{weight resolution theorem} and the first paragraph respectively. 
Hence we obtain the result.
\end{proof}

\sn
Recall the definition of $\cM_{A,?}(\ff_U;\ff_V)(p)$ 
as a subcategory of $\Cub^V(\cM^{\ff_U}_A)$ 
from \ref{lemdf:cMAfTfsp}. 

\begin{cor}
\label{cor:Devissage new} 
For any disjoint decomposition $S=U\coprod V$, 
and integer $p \geqq \#U$, 
the inclusion functor 
$\cM_{A,\red}(\ff_U;\ff_V)(p) \rinc 
\cM_A(\ff_U;\ff_V)(p)$
induces a homotopy equivalence on $K$-theory: 
$$K(\cM_{A,\red}(\ff_U;\ff_V)(p)) \to
K(\cM_A(\ff_U;\ff_V)(p)).$$
\end{cor}

\begin{proof}[\bf Proof]
The inclusion functors 
$\cM_{A,\red}^{\ff_{U\coprod T}}(p+\# T) \rinc 
\cM_{A}^{\ff_{U\coprod T}}(p+\# T)$ for any $T\in \cP(V)$ and 
$\cM_{A,\red}(\ff_U;\ff_V)(p) \rinc 
\cM_A(\ff_U;\ff_V)(p)$ and the exact functor $\lambda$ 
in \ref{cor:main theorem} $\mathrm{(2)}$ 
yield the commutative diagram below: 
$${\footnotesize{\xymatrix{
K(\cM_{A,\red}(\ff_U;\ff_V)(p)) \ar[r] \ar[d]_{K(\lambda)}^{\wr} &
K(\cM_A(\ff_U;\ff_V)(p)) \ar[d]^{K(\lambda)}_{\wr}\\
\underset{T\in\cP(V)}{\prod}
K(\cM_{A,\red}^{\ff_{U\coprod T}}(p+\# T)) \ar[r]^{\!\!\!\sim} &
\underset{T\in\cP(V)}{\prod}
K(\cM_A^{\ff_{U\coprod T}}(p+\# T)).
}}}$$
Here the vertical maps and the horizontal bottom map are 
homotopy equivalence by \ref{cor:main theorem} $\mathrm{(2)}$ and 
\ref{cor:Devissage 2} respectively. 
Hence we obtain the result.
\end{proof}

\sn
Recall the definitions of $\Kos_A^{\ff_S}$ from \ref{df:Koszul cube df} 
and 
$\Kos_{A,\red}^{\ff_S}$ and $\tq$ from \ref{lemdf:cMAfTfsp}. 

\begin{cor}
\label{cor:Devissage 3}
The canonical inclusion functor 
$\iota:\Kos_{A,\red}^{\ff_S} \rinc \Kos_A^{\ff_S}$ 
induces the following homotopy equivalences on $K$-theory: 
\begin{align*}
K(\Kos_{A,\red}^{\ff_S}) &\to 
K(\Kos_A^{\ff_S})\\
K(\Kos_{A,\red}^{\ff_S};\tq) &\to 
K(\Kos_A^{\ff_S};\tq).
\end{align*}
\end{cor}

\begin{proof}[\bf Proof] 
If $S=\emptyset$, 
then $\Kos_A^{\ff_S}=\Kos_{A,\red}^{\ff_S}=\cP_A$. 
In this case, the assertion is trivial. 
We assume $\# S\geq 1$. 
Since we have the equality $\mathrm{(\ref{equ:Kos as MAfUfV})}$ 
in \ref{rem:cMAfTfsp} $\mathrm{(3)}$, 
the first homotopy equivalence is just the special case 
$U=\emptyset$, $p=0$
of 
\ref{cor:Devissage new}. 
Let us consider the commutative diagram 
induced from the exact functor 
$\Homo_0\Tot:(\Kos_{A,?}^{\ff_S},\tq) \to \cM_{A,?}^{\ff_S}(\# S)$ 
defined in \ref{cor:main theorem} $\mathrm{(1)}$ 
and the inclusion functors 
$\Kos_{A,\red}^{\ff_S}\rinc \Kos_{A}^{\ff_S}$ and 
$\cM_{A,\red}^{\ff_S}(\# S)\rinc \cM_A^{\ff_S}(\# S)$ 
below:
$${\footnotesize{\xymatrix{
K(\Kos_{A,\red}^{\ff_S};\tq) \ar[r] \ar[d]_{K(\Homo_0\Tot)}^{\wr}  &
K(\Kos_{A}^{\ff_S};\tq) \ar[d]^{K(\Homo_0\Tot)}_{\wr}\\
K(\cM_{A,\red}^{\ff_S}(\# S)) \ar[r]^{\ \ \sim} & 
K(\cM_A^{\ff_S}(\# S)).
}}}$$
The vertical lines above are homotopy equivalences 
by \ref{cor:main theorem} $\mathrm{(1)}$. 
The bottom horizontal line above is also a homotopy equivalence by \ref{cor:Devissage 2}. 
Hence we obtain the second homotopy equivalence.
\end{proof}

\begin{proof}[\bf Proof of Corollary~\ref{cor:gercon}]
By \ref{rem:regularclosedimmerion} and \ref{thm:HM10}, 
we have the homotopy equivalences 
$$
\varinjlim_{\fg_S} K(\cM_A^{\fg_S}(\# S)) \isoto 
K(\cM_A^{\# S}(\# S)) \isoto 
K(\cM_A^{\# S})
$$
where $\fg_S$ runs over $A$-regular sequences indexed by $S$. 
Therefore Gersten's conjecture for $A$ 
is equivalent to the following assertion: 

\sn
{\it For any $A$-regular sequence $\{g_s\}_{s\in S}$ in $A$, 
the inclusion functor 
$\cM^{\fg_S}_{A}(\# S) \rinc \cM^{\# S-1}_A$ 
induces the zero maps on $K$-groups.}

\sn
Since $A$ is local, every $A$-regular sequence is 
an $A$-sequence by \ref{ex:A-sequence} $\mathrm{(3)}$. 
Fix an $A$-regular sequence $\ff_S=\{f_s\}_{s\in S}$ in $A$ 
and write $j$ for the inclusion functor 
$\cM_A^{\ff_S}(\# S) \rinc \cM_A^{\# S -1}$. 
Then let us consider the commutative diagram below:
$${\footnotesize{\xymatrix{
K(\Kos_{A,\red}^{\ff_S}) \ar[r]^{\textbf{I}} \ar@{->>}[d] & 
K(\Kos_{A}^{\ff_S}) \ar@{->>}[d] \ar[rrd]^{K(j\Homo_0\Tot)}\\
K(\Kos_{A,\red}^{\ff_S};\tq) \ar[r]_{\textbf{I}} & 
K(\Kos_{A}^{\ff_S};\tq) \ar[r]_{\textbf{II}}^{K(\Homo_0\Tot)} & 
K(\cM_A^{\ff_S}(\#S)) \ar[r]_{K(j)} & 
K(\cM_A^{\# S-1}).
}}}$$
Here the maps \textbf{I} and \textbf{II} are homotopy equivalences 
by \ref{cor:Devissage 3} and 
\ref{cor:main theorem} $\mathrm{(1)}$ respectively and 
the vertical maps are (split) epimorphisms by \ref{cor:main theorem} 
$\mathrm{(3)}$. 
Hence $K(j)$ is trivial if and only if 
the composition 
$K(j\Homo_0\Tot):K(\Kos_{A,\red}^{\ff_S}) \to K(\cM_A^{\# S-1})$ 
is trivial. 
Therefore we get the desired result. 
\end{proof}

\end{document}